\documentclass{article}

\usepackage[utf8]{inputenc}
\usepackage{geometry}
\usepackage{amssymb}
\usepackage{amsmath}
\usepackage{enumitem}
\usepackage{graphicx}
\usepackage{subcaption}
\usepackage{float}
\usepackage{amsthm}
\usepackage{algorithm}
\usepackage{algpseudocode}
\usepackage[colorlinks,linkcolor=orange, anchorcolor=blue, citecolor=blue]{hyperref}
\usepackage{bbm}
\usepackage{xcolor}
\usepackage{smile}
\usepackage{comment}
\usepackage{tabularx} 
\usepackage{booktabs}   
\usepackage[table]{xcolor}
\definecolor{rowgray}{gray}{0.85} 
\usepackage{array}
\definecolor{colgray}{gray}{0.85}
\newcolumntype{G}{>{\columncolor{colgray}}c} % shaded column type
\allowdisplaybreaks[4]
\usepackage[table]{xcolor}

\DeclareMathOperator*\SSS{\mathcal{S}}

\newcommand{\cms}[1]{{\color{purple} \small \sf $\ll$Minshuo: {#1}$\gg$}}

\newcommand{\be}{\begin{equation}}
\newcommand{\ee}{\end{equation}}

\theoremstyle{plain}
\newtheorem{thm}{Theorem}[section]
\newtheorem{lem}[thm]{Lemma}
\newtheorem{prop}[thm]{Proposition}
\newtheorem{cor}{Corollary}

\theoremstyle{definition}

\newtheorem{assumption}{Assumption}

\theoremstyle{remark}
\newtheorem{rem}{Remark}

\usepackage{setspace}
\usepackage{natbib}

\title{Parameter-Efficient Subspace Optimization for LLM Fine-Tuning}
\author{
  Yuchen Lou\thanks{Department of Industrial Engineering and Management Sciences, Northwestern University. 
  Emails: \texttt{yuchenlou2026@u.northwestern.edu}, 
  \texttt{zeqiye2029@u.northwestern.edu}, 
  \texttt{minshuo.chen@northwestern.edu}} \quad
  Zeqi Ye\footnotemark[1] \quad
  Minshuo Chen\footnotemark[1]
}
\date{\today}

\begin{document}

\maketitle

\begin{abstract}
    % This paper develops a new perspective on parameter-efficient fine-tuning (PEFT) for LLMs, inspired by the classical theory of subspace minimization. 
    % We introduce a unifying framework, {\bf P}arameter-{\bf E}fficient {\bf S}ubspace {\bf O}ptimization ({\bf PESO}), which not only recovers many existing methods such as LoRA but also bridges them with the principled algorithmic and theoretical foundations of subspace optimization. 
    % This connection highlights a natural ``exploration--exploitation'' view of subspace methods, guiding the design of new algorithms that achieve strong convergence performance while still preserving memory efficiency.  
    % We further instantiate the framework into a practical algorithm named {PESO-LoRA}, based on LoRA-type parameterization. 
    % Importantly, our algorithm establishes the convergence in the full-parameter space while maintaining the empirical efficiency of LoRA parameterization, resolving a critical gap of LoRA variants where {low-rank updates} lack such guarantees.
    % Our algorithm achieves notable improvements over existing methods on standard benchmarks.
    This paper develops a new perspective on parameter-efficient fine-tuning (PEFT) for LLMs, inspired by classical subspace minimization.
    We introduce a unifying framework, {\bf P}arameter-{\bf E}fficient {\bf S}ubspace {\bf O}ptimization ({\bf PESO}), which recovers existing methods such as LoRA and connects them to the principled algorithmic and theoretical foundations of subspace optimization.
    This connection highlights a natural ``exploration--exploitation'' view of subspace methods, guiding the design of new algorithms that achieve strong convergence performance while still preserving memory efficiency. 
    We instantiate the framework into a practical algorithm, {\bf PESO-LoRA}, based on a LoRA-type parameterization.
    Importantly, we provide convergence guarantees stated in the \emph{full-parameter space} for the induced update, addressing a key limitation of LoRA-style analyses that only track low-dimensional factors.
    Empirically, {PESO-LoRA} improves over strong PEFT baselines on standard fine-tuning benchmarks.
\end{abstract}

\section{Introduction}
%\yl{[Todo list on Sun: 1. Full revised intro using the new structure; 2. consider where to put the synthetic example and the general pseudo code; 3. revised convergence; 4. appendix details; 5. trim down to 9 pages; 6. (can be later) more references to add.]}
Training deep neural networks is the cornerstone of modern AI, powering the success of large-scale foundation models such as Large Language Models (LLMs) \citep{brown2020language}.  
At the core, it reduces to solving a high-dimensional optimization problem over weight matrices:
%\zy{[Shall we even mention pre-training at the very start?]} \yl{[Agreed; i changed it.]}
% How to efficiently train Large Language Models (LLMs) \citep{xxx} is one of the central questions in modern deep learning.
% Formally, training can be expressed as the following optimization problem:
\begin{equation}\label{opt_original}
\textstyle \Delta W^*:=\arg\min_{\Delta W} \,\, \ell\big(W_0 + \Delta W\big),
\end{equation}
where $\ell(\cdot)$ is the loss function, $W_0$ is the initialization, and $\Delta W$ the increment.  
In practice,~\eqref{opt_original} is typically solved by first-order methods such as Adam \citep{kingma2014adam} and AdamW \citep{loshchilov2017decoupled}, which are the workhorses of large-scale training.
However, these methods require storing additional optimizer states (e.g., momentum and velocity), and for LLMs this overhead places enormous pressure on memory resources, making parameter-efficient strategies appealing.

% In the setting of Parameter-Efficient Fine-Tuning (PEFT) \cms{PEFT appears in a sudden}\citep{}, $W_0$ corresponds to pretrained weights obtained from large-scale pretraining, and $\ell$ is the task-specific loss associated with a downstream application. 
% This work focuses on PEFT, while noting that many of the methodologies developed here naturally extend to the pretraining regime as well. \yl{[Broad motivation]}

%[One Version of Introduction: first describe DL benchmarks, their issues and raise a central question for LLM society interest, then raise our intention to connect with subspace methods, and finish with problem set-up using intrinsic dimension. Lighter start and less math, may lack accurate justification of our focus.]

In the realm of fine-tuning, we often have limited labeled data for a downstream task but still wish to adapt the pretrained weights effectively and efficiently.
Therefore, updating the entire parameter set is memory-intensive.
This motivates the study of Parameter-Efficient Fine-Tuning (PEFT) methods \citep{han2024parameter,houlsby2019parameter,hu2022lora}, where optimization is restricted to a smaller set of parameters initialized from pretrained weights.
In other words, $W_0$ denotes weights obtained from a large-scale pretraining phase, and $\Delta W$ is not updated freely but instead follows an efficient parameterization that constrains the search space.

% \begin{figure*}
%     \centering
%     \makebox[\linewidth][c]{%
%         \begin{subfigure}[t]{0.33\linewidth}
%             \centering
%             \includegraphics[width=\linewidth]{Figures/real_data_full_lora_restart.pdf}
%             %\caption{LLM}
%             \label{fig:real_data}
%         \end{subfigure}%
%         \begin{subfigure}[t]{0.32\linewidth}
%             \centering
%             \includegraphics[width=\linewidth]{Figures/loss_peso_demo.pdf}
%             \label{fig:trajectory}
%         \end{subfigure}\hspace{1em}
%         \begin{subfigure}[t]{0.32\linewidth}
%             \centering
%             \includegraphics[width=\linewidth]{Figures/synthetic_matrix_full_lora_restart.pdf}
%             %\caption{Synthetic Example}
%             \label{fig:synthetic_matrix}
%         \end{subfigure}
%     }
%     %\vspace{-0.5em}
%         \vskip -0.1in
%     \caption{Comparison of full-parameter tuning, LoRA, and our method (PESO-LoRA). 
%     Left: MetaMathQA.  
%     Middle:  synthetic example $\min_{W}\|W-M\|_F^2$ with $M = 10 \cdot \operatorname{diag}(1,\dots,1,0,\dots,0)$ ($r{+}1$ ones); see Appendix~\ref{sec:synthetic}.
%     Right: optimization trajectories.
%     PESO-LoRA bridges the loss gap of LoRA while preserving memory and computation efficiency.}
%         \vskip -0.20in
%     \label{fig:examples}
% \end{figure*}

\begin{figure*}[t]
    \centering
    \begin{minipage}[b]{0.58\linewidth}
        \centering
        \begin{subfigure}{\linewidth}
            \centering
            \includegraphics[width=\linewidth]{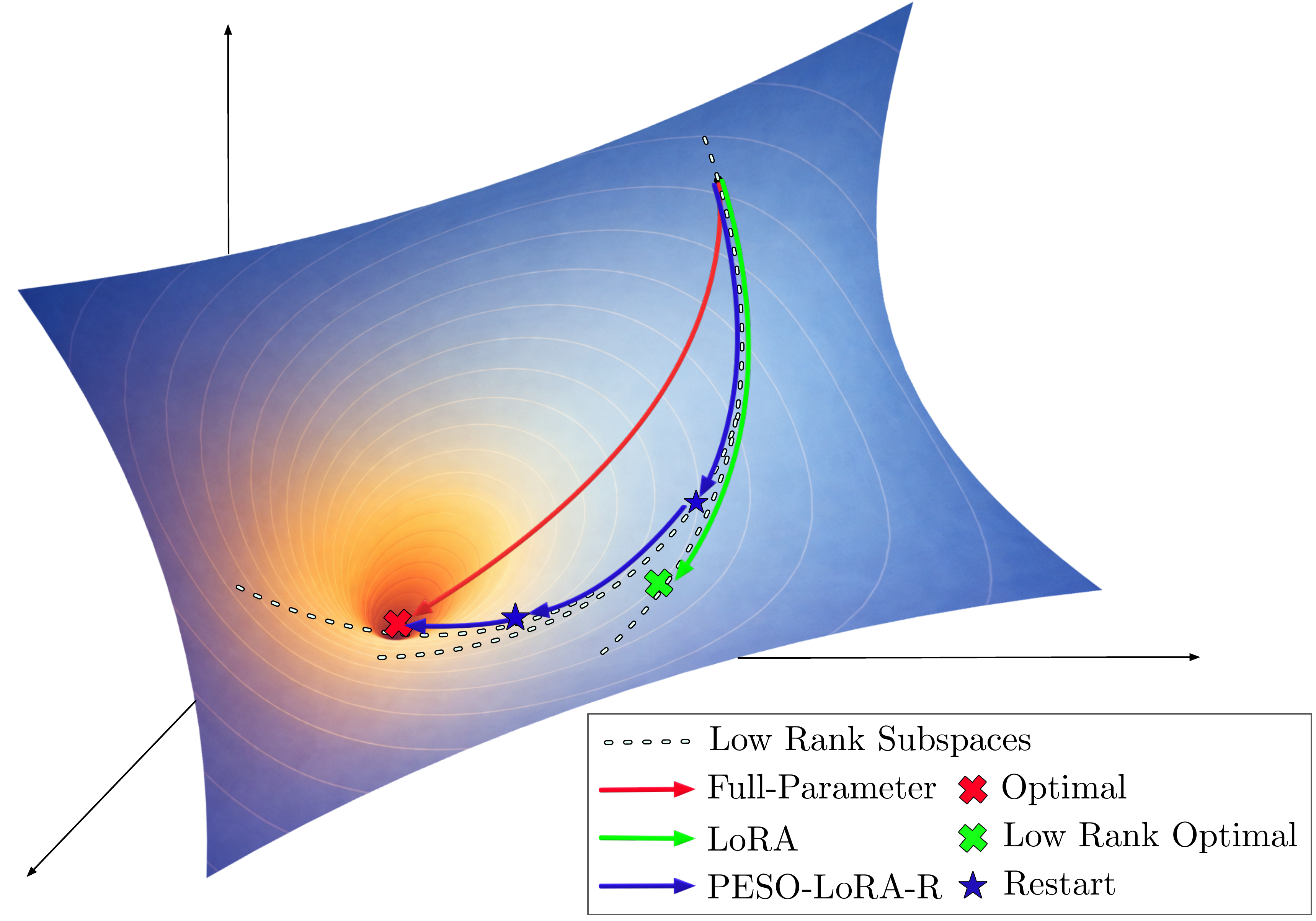}
            \label{fig:trajectory}
        \end{subfigure}
    \end{minipage}
    \hfill
    \begin{minipage}[b]{0.41\linewidth}
        \centering
        \begin{subfigure}{\linewidth}
            \centering
            \includegraphics[width=0.9\linewidth]{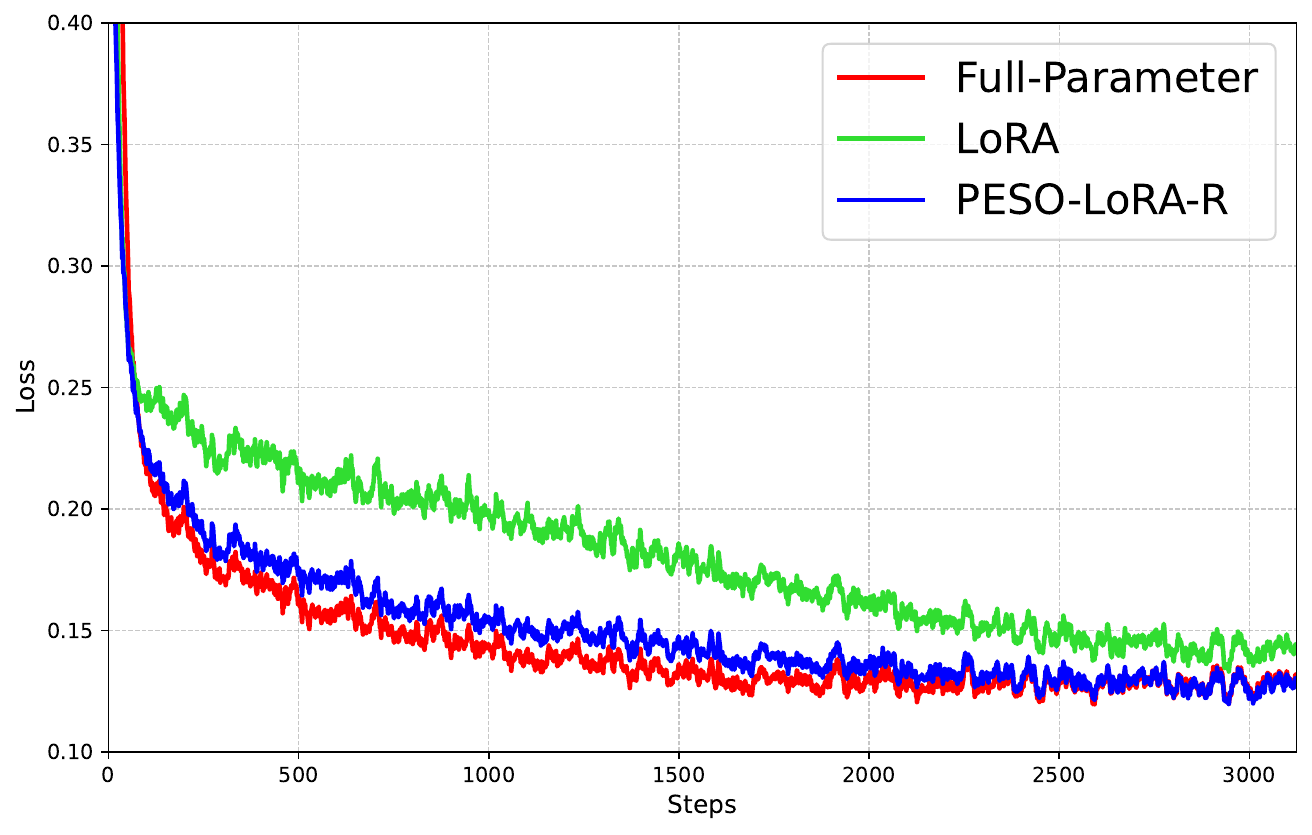}
            \label{fig:real_data}
        \end{subfigure}

        \begin{subfigure}{\linewidth}
            \centering
            \includegraphics[width=0.9\linewidth]{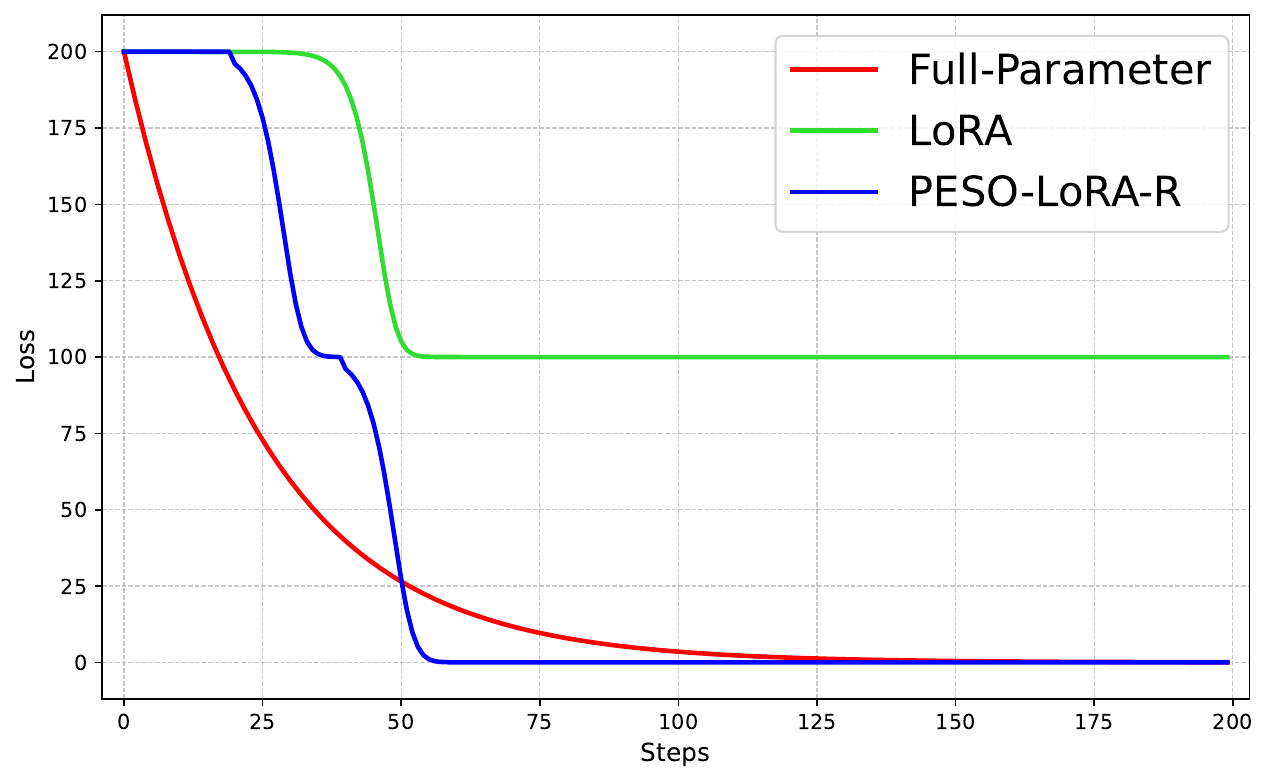}
            \label{fig:synthetic_matrix}
        \end{subfigure}
    \end{minipage}

    %\vspace{-0.5em}
    \caption{Comparison of full-parameter tuning, LoRA, and our method (PESO-LoRA). 
    Left: optimization trajectories. 
    Top Right: MetaMathQA. 
    Bottom Right: synthetic example $\min_{W}\|W-M\|_F^2$ with $M = 10 \cdot \operatorname{diag}(1,\dots,1,0,\dots,0)$ ($r{+}1$ ones); see Appendix~\ref{sec:synthetic}. 
    PESO-LoRA bridges the loss gap of LoRA while preserving memory and computation efficiency.}
    %\vspace{-1.5em}
    \label{fig:examples}
\end{figure*}

A popular PEFT method is low-rank adaptation (LoRA, \cite{hu2022lora}), where matrices in $\Delta W$ are expressed as the product of two low-rank factors.  
LoRA has shown strong empirical success, reducing memory costs while achieving competitive downstream performance.  
However, it suffers from two key limitations:  
1) performance often lags behind full-parameter fine-tuning (Figure~\ref{fig:examples}, top right: MetaMathQA);  
2) theoretical guarantees are limited, with convergence typically shown only for the low-rank factors (Figure~\ref{fig:examples}, bottom right: a synthetic example illustrating LoRA’s potentially unbounded loss gap).  
To address these issues, many LoRA variants \citep{hayou2024lora+,wang2024loraga,wang2024lora,zhang2023adalora,zhang2025one} have been proposed, yet they largely inherit the same shortcomings and leave the following fundamental question open:
\begin{center}
\it Can we design fine-tuning methods that maintain the {practical performance} of LoRA while still enjoying the convergence and optimality of full-parameter fine-tuning?
\end{center}
%% Review:  memory footprint ---> practical performance

To address this question, we reveal an inherent connection between {PEFT} and the classical idea of \textbf{subspace minimization}, a long-standing nonlinear optimization strategy dating back to \citet{conn1994iterated,cragg1969study}. 
The central philosophy is to decompose a large-scale problem like~\eqref{opt_original} into {\bf iterative, simpler} subproblems constrained to carefully chosen subspaces. 
This view resonates naturally with modern PEFT methods, which restrict updates to structured low-rank forms for better efficiency. 
Interestingly, subspace minimization historically received less attention in the optimization society, since full-parameter information were often affordable in traditional applications.
However, it is especially well suited to LLM training, where {\it massive dimensionality} calls for {\it memory-efficient} methods.
%we believe the stringent \textit{memory constraints} of LLM training now make such ideas highly relevant again.

%Given the growing sizes of language models\cms{move this forward before PEFT}, one key practical restriction of training LLMs is the limited memory budget.

Formally, we build on the notion of \textit{intrinsic dimensionality} in LLM training \citep{aghajanyan2020intrinsic,li2018measuring}, recognized in \citep{hu2022lora} as the origin of LoRA: there exists a {dimension-lifting} map $\mathcal{M}: \mathbb{R}^d \to \mathbb{R}^{m \times n}$, with $d \ll m \times n$, such that the optimal solution $\Delta W^*$ of~\eqref{opt_original} satisfies
\begin{equation}\label{eq:intrinsic_dim}
\textstyle
    \Delta W^* \approx \mathcal{M}(\xi^*), 
    \, \xi^* := \arg\min_{\xi \in \mathbb{R}^d} \ell(W_0 + \mathcal{M}(\xi)).
\end{equation}
{Here, $d$ stands for the number of trainable parameters}, and this characterization implies that it suffices to optimize within the reduced space defined by {the image of} $\mathcal{M}$ to approximate $\Delta W^*$.  
For clarity, we focus on a single weight matrix $\Delta W \in \mathbb{R}^{m \times n}$ (multi-layer extensions are straightforward) and represent $\xi$ as a $d$-dimensional vector. This is without loss of generality, since tensor parameters can always be flattened via vectorization into an isomorphic Euclidean space.  
For example, LoRA adopts the simple form $\mathcal{M}(A,B) = AB$ with $A \in \mathbb{R}^{m \times r}$, $B \in \mathbb{R}^{r \times n}$, and $d = (m+n)r$.
{However, it remains unclear whether such a simple $\mathcal{M}$ is sufficient to capture the complexity of LLM training dynamics.}
%though there is little justification that the true $\mathcal{M}$ should take such a simple form.  

Our framework approximates $\mathcal{M}$ adaptively through a \textit{sequential subspace approximation}, providing a more effective capture of~\eqref{eq:intrinsic_dim}.
We construct a sequence of maps $\{\mathcal{M}_k\}$, each with a simple representation {and $\mathbb{R}^d \to \mathbb{R}^{m \times n}$},
\begin{equation}\label{eq:intrinsic_subspace}
\begin{array}{l}
\Delta W^* \approx \sum_k \mathcal{M}_k(\xi_k^*), \\
\xi_k^* := \arg\min_{\xi}
\ell\!\left(
W_0 + \sum_{i=1}^{k-1} \mathcal{M}_i(\xi_i^*)
+ \mathcal{M}_k(\xi)
\right).
\end{array}
\end{equation}
% \begin{align}\label{eq:intrinsic_subspace}
% \textstyle
%     &\Delta W^* \approx \sum_k \mathcal{M}_k(\xi_k^*),
%     \notag\\
%     & \xi_k^*:=\arg\min_{\xi \in \mathbb{R}^d} \ell(W_0 + \sum_{i=1}^{k-1} \mathcal{M}_i(\xi_i^*)+\mathcal{M}_k(\xi)).
% \end{align}
Each {image of} $\mathcal{M}_k$ approximates a subspace and $\xi_k^*$ is its low-dimensional coordinate. 
In essence, the complexity of $\mathcal{M}$ is captured by a sequence of piecewise-linear subspaces. 
This philosophy parallels classical approximation schemes in numerical analysis such as finite element methods \citep{bathe2006finite}.

Guided by this perspective, we develop a principled framework for PEFT grounded in subspace minimization, named {\bf P}arameter-{\bf E}fficient {\bf S}ubspace {\bf O}ptimization ({\bf PESO}).
A key insight is to view the problem~\eqref{eq:intrinsic_subspace} through an \textbf{exploration–exploitation} lens: 
\emph{exploration} designs new subspaces that capture full gradient information, while \emph{exploitation} optimizes efficiently within the current subspace. 
This resolves LoRA’s two central limitations: lack of full-parameter convergence guarantees and inefficiency from rigid low-rank parameterization; see Figure~\ref{fig:examples}.
%As illustrated in Figure~\ref{fig:examples}, our method PESO-LoRA bridges this gap.

% \subsection{Contributions}
\noindent {\bf Contributions}.
Our contributions can be summarized at three levels.
Although our focus is on PEFT, many of the ideas developed here naturally extend to pre-training.
%\cms{maybe too much for a conference paper. I will merge perspective into algorithmic. Empirical level is more like validation instead of a contribution.}
\paragraph{I. Perspective Level.} %\vspace{-0.5em} 
We introduce a novel framework PESO for memory-efficient training inspired by classical subspace minimization \citep{conn1994iterated}, unifying existing PEFT approaches such as LoRA variants \citep{hu2022lora,wang2024loraga,wang2024lora,zhang2023adalora,zhang2025one} and GaLore \citep{zhao2024galore}. 
This framework allows us to explore the rich algorithmic techniques in subspace methods, providing systematic guidance to improve memory-efficient methods. 
In particular, we highlight two complementary directions: \textit{exploration} of new subspaces through information from the full gradient, and \textit{exploitation} of the current subspace via streaming SVD representations.

\paragraph{II. Theoretical Level.} %\vspace{-0.5em}
% Our exploration mechanism, \textit{full gradient restart}, enables the framework to effectively guide training dynamics.
% The resulting algorithm is, to our knowledge, the first method for {LLM fine-tuning that combines the practical effectiveness of LoRA-style designs with a provable convergence to full-parameter optimality, without requiring additional assumptions such as explicit low-rankness of the solution.}
Our exploration mechanism, \textit{full gradient restart}, enables the framework to effectively guide training dynamics.
The resulting algorithm is tailored to LLM fine-tuning and combines the empirical efficiency of LoRA-style designs with convergence guarantees stated in the \emph{full-parameter space} for the induced update.
This aspect is largely absent from existing LoRA-style analyses, which typically track only the low-dimensional factors.

\paragraph{III. Empirical Level.} %\vspace{-0.5em}
Guided by our PESO framework, we show that two practical instantiations of our framework---{\tt PESO-LoRA-R} and {\tt PESO-LoRA-T}---achieve improved performance while preserving the memory efficiency of state-of-the-art PEFT methods across benchmarks such as GLUE, mathematical reasoning, code generation, and general instruction tuning.
%These results demonstrate both the effectiveness and flexibility of the proposed framework.

% \subsection{Related Work}
\noindent {\bf Related Work}.
LoRA \citep{hu2022lora} is perhaps the most widely known PEFT method, and numerous variants have been proposed for better performance. 
For instance, LoRA+ \citep{hayou2024lora+} introduces imbalanced learning rates; 
PiSSA \citep{meng2024pissa} proposes an initialization from SVD of $W_0$; 
and AdaLoRA \citep{zhang2023adalora} maintains an adaptive SVD-based low-rank representation. 
Other extensions focus on gradient scaling \citep{tastan2025loft,zhang2024riemannian}. 
More recent work leverages information from the full gradient: LoRA-GA \citep{wang2024loraga} and LoRA-Pro \citep{wang2024lora} propose memory-efficient gradient approximations, and LoRA-One \citep{zhang2025one} employs the SVD of the full gradient for initialization.

One closely related line of work studies memory-efficient training via projection or subspace updates in the pre-training regime.
GaLore \citep{zhao2024galore} constructs the projection subspace using a top-$r$ SVD of the full gradient, while Fira \citep{chen2024fira} augments this with a gradient-correction step to reduce projection bias. 
SARA \citep{zhang2025breaking} employs importance sampling to choose subspaces based on the SVD spectrum of the full gradient. 
APOLLO \citep{zhu2024apollo} instead uses randomized projections to define the subspace, and SubTrack++ \citep{rajabi2025subtrack++} explores Grassmannian subspaces for improved tracking.
While these pre-training methods can be adapted to PEFT, they are not designed to exploit the strong intrinsic dimensionality~\eqref{eq:intrinsic_dim} in PEFT tasks and can underperform LoRA-style parameterizations; see Table~\ref{tab:subspace_comparison}.
In contrast, PESO builds on classical subspace minimization \citep{conn1994iterated}, embedding LoRA-type structures and viewing fine-tuning as sequential subproblems that iteratively capture intrinsic dimensionality \eqref{eq:intrinsic_subspace}.

Convergence guarantees for PEFT algorithms remain scarce, and existing results typically address only the low-dimensional parameters \citep{jiang2024unified}.
A closely related line of work studies \emph{randomized subspace descent} \citep{kozak2019stochastic}, including RSO \citep{chen2025memory} and GoLore \citep{he2024subspace}, which construct subspaces via randomized projections and provide convergence guarantees in the full-parameter space.
Their exploration strategies differ from ours and they typically do not encode PEFT-specific LoRA structures.
Other approaches require additional structural conditions, e.g., \citet{liang2024memory} assumes the projection subspace has full column rank (unrealistic when $r<m$) and LDAdam \citep{robert2024ldadam} assumes a strict contraction.
Closer to LoRA, \citet{jang2024lora} analyze LoRA only in the NTK regime, and RAC-LoRA \citep{malinovsky2024randomized} proposes a randomized asymmetric-chaining variant, a different exploration mechanism from ours.
%which constrain updates to $W_{k+1} \gets W_k - \eta_k P_k P_k^\top G_k$, where $P_k$ is the projection matrix, $\eta_k$ the learning rate, and $G_k$ the full gradient. 
%These approaches establish convergence in the full-parameter space, but under extra structural assumptions.
% Likewise, \citet{chen2025memory,he2024subspace,kozak2019stochastic} rely on random projection theory, assuming $\mathbb{E}[P P^\top] = I_m$ and $P^\top P = I_r$, conditions not needed in our analysis.

Subspace minimization is a classical theme in nonlinear optimization \citep{conn1994iterated,cragg1969study,yuan2014review}. 
It was historically overshadowed by full-parameter algorithms such as L-BFGS \citep{liu1989limited} and conjugate gradient methods \citep[Ch. 5]{nocedal1999numerical}, since many traditional applications could afford storing full gradients and quasi-Newton pairs. 
More recently, however, subspace-based strategies have re-emerged in large-scale derivative-free optimization \citep{cartis2023scalable,dzahini2024stochastic,menickelly2024augmenting,nozawa2025zeroth,zhang2025scalable}, where gradients are unavailable and low-dimensional surrogates are crucial.

\section{PESO: A Framework From Subspace Minimization}

%\cms{A general comment: Most conference paper readers or reviewers are impatient. Therefore, we may want to avoid large bulk of paragraphs. We can use bullets, bold fonts and structured presentation to allow reviewers skipping without losing track of the paper.}
%\cms{Add a sentence here to summarize the whole section.}
In this section, we provide a novel perspective of PEFT methods with insights from subspace minimization.
We summarize an algorithmic framework {\bf P}arameter-{\bf E}fficient {\bf S}ubspace {\bf O}ptimization ({\bf PESO}) in Algorithm~\ref{alg:generic_subspace}, and discuss how it unifies many benchmarks such as LoRA and GaLore.

%\subsection{Generic Representation \cms{this title is a bit vague. If don't have a good name, we can start a subsection for exploration and exploitation.}}
To build an iterative scheme, a central question is how to represent the weight $W$ at each iteration using low-dimensional representation $\xi$.
{In~\eqref{eq:intrinsic_subspace}, the optimization is expressed through evolving subspaces.  
At iteration $k$, we define the {\it anchored state} $\tilde W_k := W_0 + \sum_{i=1}^{k-1} \mathcal{M}_i(\xi_i^*)$ to encode historical progress, and represent}
\begin{equation}\label{eq:representation_generic}
    W_k = \tilde W_k + \mathcal{M}_k(\xi_k).
\end{equation}
% where $\tilde{W}_k \in \mathbb{R}^{m \times n}$ serve as an anchored state encoding historical progress, consistent with~\eqref{eq:intrinsic_subspace} by $\tilde W_k=W_0 + \sum_{i=1}^{k-1} \mathcal{M}_i(\xi_i^*)$.
%\cms{what is accumulated historical information? How can we get $\widetilde{W}_k$?}\yl{[This is a hard point. I need this since it is basically the info to record once subspace is updated; while in LoRA people just set it to be $W_0$ without changing (because they only use warm start). Things can only become clearer after we introduce warm start/restart and the framework. i am also debating if I should first present the connections to exisiting methods, or first present the framework in an abstract sense.]}\cms{Can you try to rephrase LoRA in this general framework so that ppl know what is $\tilde{W}_k$.}\yl{[I tried, but I found it may be harder to understand the full picture before having the framework first. My plan is now to at least finish revising the framework and then decide which to go first. now this part is in the next subsection; although not done as well.]}\cms{ok!}
Following the design of subspace minimization, PESO considers each $\mathcal{M}_k$ to admit a simple image in the form of a subspace: $\mathcal{S}_k := \{\mathcal{M}_k(\xi) : \xi \in \mathbb{R}^d\}$.
%\cms{I think this assumption comes from subspace optimization and we adopt it. Another thing is shall we say LoRA fixes $\tilde{W}_k$ to $W_0$ first and then talk about exploration and exploitation?}

Under representation~\eqref{eq:representation_generic}, the evolution of $W_k$ can be viewed as three complementary operations:
1) {\textbf{exploration}:} updating $\mathcal{M}_k$ to select a new subspace $\mathcal S_k$, {(line~\ref{line:explore} of Algorithm~\ref{alg:generic_subspace})}
2) {\textbf{exploitation}:} optimizing $\xi_k$ within the current subspace {(line~\ref{line:exploit1}-\ref{line:exploit2} of Algorithm~\ref{alg:generic_subspace})}, and
3) updating $\tilde W_k$ to absorb past progress into the anchored weights.
%The first corresponds to \textbf{exploration} of new directions, the second to \textbf{exploitation} of the chosen subspace, and the third to maintaining consistency across successive subspaces.
%\cms{put exploration and exploitation tags directly into 1) and 2)? We can mention Algorithm 1 box here?}
{These operations} mirror the classical paradigm of \textbf{subspace minimization} \citep{conn1994iterated}, where a large-scale problem is solved by iteratively: (i) constructing a subspace based on local information such as gradients, (ii) solving a reduced subproblem within that subspace, (iii) updating the iterate to incorporate the subspace solution. %\cms{not sure if this paragraph is needed}

In our design, exploration and exploitation directly parallel subspace selection and subproblem optimization, while the anchored state $\tilde W_k$ retains progress from earlier subspaces. 
In LoRA, $\tilde W$ is fixed at $W_0$, confining progress to the active subspace. 
In contrast, updating $\tilde W$ absorbs accumulated contributions back into the parameter space, giving rise to two distinct exploration strategies: \emph{warm-start} and \emph{restart}, which we detail below in Section~\ref{sec:explore_exploit}.

Leveraging this connection to subspace minimization, we present our generic framework PESO in Algorithm~\ref{alg:generic_subspace}.
It is important to note that, by selecting corresponding parameterization of $\mathcal{M}_k$, {\tt UpdateM}, and {\tt Opt}, we are able to recover a variety of existing benchmarking methods in parameter-efficient training; see representatives in Table~\ref{tab:methods}.
We also remark that Algorithm~\ref{alg:generic_subspace} is equivalent to the classical two-loop subspace minimization scheme \citep{conn1994iterated}, which we defer to Appendix~\ref{sec:subspace_review} in Algorithm~\ref{alg:subspace_algo}.
%\cms{If I am correct, the indexing of $\mathcal{M}$ and $\tilde{W}$ is broken. Because you only update those two every $K$ iterations, but the index $k$ will increase, leading to undefined $\mathcal{M}_{k}$ and $\tilde{W}_k$ when $k$ is not a multiple of $K$.}
%\cms{Presenting a figure to demonstrate the algorithm is helpful here.}
%\vspace{-0.2em}
\begin{algorithm}[ht]
\caption{{\bf PESO}: Generic Framework of {\bf P}arameter-{\bf E}fficient {\bf S}ubspace {\bf O}ptimization}
\label{alg:generic_subspace}
\begin{algorithmic}[1]
    \Require Initialization $W_0\in\mathbb{R}^{m\times n}$, $\xi_0\in\mathbb{R}^d$, and $\mathcal{M}_0$; 
    an algorithmic subroutine \texttt{UpdateM}, 
    an optimizer \texttt{Opt}, frequency $K$.
    \State Set $k \gets 1$ and $\tilde W_0\gets W_0$.
    \While{stopping criteria not satisfied}
        \State {$(\mathcal{M}_{k}, \tilde W_{k}) \gets (\mathcal{M}_{k-1}, \tilde W_{k-1})$.}
        \If{$k-1 \bmod K = 0$} \Comment{{\it Explore} new $\mathcal{S}_{k}$}
            \State $(\mathcal{M}_{k},\tilde W_{k})$ $\gets {\texttt{UpdateM}}(\mathcal{M}_{k-1},\tilde W_{k-1})$.\label{line:explore}
        \EndIf
        \State $\Delta\xi_{k} \gets {\texttt{Opt}}(\xi_{k-1}, \mathcal{M}_{k})$ \label{line:exploit1}\Comment{{\it Exploit} current $\mathcal{S}_k$}
        \State $\xi_{k} \gets \xi_{k-1} + \Delta \xi_{k}$. \label{line:exploit2}
        %\State $W_{k+1}\gets \tilde W_{k+1}+\mathcal{M}_{k+1}(\xi_{k+1})$.
        \State $k \gets k+1$.
    \EndWhile
\end{algorithmic}
\end{algorithm}
%\vspace{-0.8em}

{\subsection{Subspace Exploration-Exploitation in PESO}\label{sec:explore_exploit}}
%\cms{subsection can start here: Subspace Exploration and Exploitation in PESO}
Now let us discuss two main components of our framework, subspace exploration and exploitation.
\paragraph{Subspace Exploration.}
Exploring new subspaces is essential for navigating the full-parameter space under memory restriction. 
Algorithm~\ref{alg:generic_subspace} carries out exploration by \texttt{UpdateM}, which updates both $\mathcal{M}_k$ and $\tilde W_k$. 
Such updates are often performed lazily every $K$ iterations, as in prior works \citep{liang2024memory,zhang2023adalora,zhao2024galore,zhu2024apollo}.

%\cms{We first define warm-start and restart and then explaining them using the previous paragraph? Again, ppl are impatient, you cannot start with a large bulk of words to drain the attention of readers.}
Depending on how much $\mathcal{M}_k$ is changed, two philosophies arise for how exploration interacts with the low-dimensional $\xi_k$: \emph{warm-start} and \emph{restart}.
These are simply two modes of \texttt{UpdateM}:
\begin{itemize}[topsep=0pt]
    \item \textbf{Warm-start.} 
    Preserve $\xi_k$ and keep $\tilde W_k$ fixed:
            %\vspace{-0.2em}
    \begin{equation}\label{eq:warm_start}
        W_{k+1} = \tilde W_k + \mathcal{M}_{k+1}(\xi_k + \Delta \xi_k).
        %\vspace{-0.8em}
    \end{equation}
    \item \textbf{Restart.} 
    Absorb the previous contribution into the baseline, $\tilde W_{k+1} \gets \tilde W_k+\mathcal{M}_k(\xi_k)$, and start the new subspace from $\xi_{\text{new}}$ (often $0$):
        %\vspace{-0.2em}
    \begin{equation}\label{eq:restart}
        W_{k+1} = \tilde W_{k+1} + \mathcal{M}_{k+1}(\xi_{\text{new}} + \Delta \xi_k).
    \end{equation}
\end{itemize}
Intuitively, warm-start provides smoother transitions when consecutive subspaces remain similar, while restart marks a new phase, useful when the optimization geometry changes sharply.  
In practice, these two modes naturally lead to two main approaches for designing {\tt UpdateM}.  
A warm-start typically updates the parameterization of $\mathcal{M}_k$ smoothly along an {\it optimization trajectory}—for example, by applying an Adam step on subspace parameters as in LoRA variants—yielding a gradually evolving subspace.  
Restart, on the other hand, often {\it reassigns} $\mathcal{M}_k$ directly using local information such as gradients.
This strategy is common in classical optimization; for example, in line search (a one-dimensional subspace method) each iteration resets the step size initialization when a new direction is chosen \citep[Ch.~3]{nocedal1999numerical}.
It is also used in LLM training, as in GaLore \citep{zhao2024galore}, which periodically resets the subspace via the SVD of the full gradient. 
Concrete examples of both approaches are summarized in Table~\ref{tab:methods}, and Section~\ref{sec:full_grad_restart} introduces a new restart scheme leveraging full gradients.

\paragraph{Subspace Exploitation.} %\vspace{-0.5em} 
Between two updates of \texttt{UpdateM}, our framework performs $K$ iterations of \texttt{Opt} within the current subspace $\mathcal{S}_k$. 
This amounts to solving the subproblem
%\vspace{-0.8em}
\begin{equation}
   \textstyle \min_{\xi \in \mathbb{R}^d} \,\, \ell(\tilde W_{k} + \mathcal{M}_{k}(\xi))
\end{equation}
approximately for $K$ steps. 
In practice, \texttt{Opt} is often chosen as Adam. %\cms{mentioning Adam seems not necessary}

The philosophy relies on a common belief in classical optimization: during an optimization procedure, once an effective subspace is identified, repeatedly exploiting it for multiple iterations improves efficiency.
This principle underlies many classical optimization methods, such as {trust-region methods \citep[Ch.~4]{nocedal1999numerical} and L-BFGS-B \citep{byrd1995limited}.}
%line search \cms{line search appears twice. Are they different or the same?}and trust-region methods \citep[Ch.~3\&4]{nocedal1999numerical}, as well as widely-used optimization packages like L-BFGS-B \citep{byrd1995limited}.

% This often involves reinitializing $\xi_k$ and incorporating the contribution of the previous subspace into $\tilde{W}_k$.
% To be specific, as $\mathcal{M}_k$ is updated to $\mathcal{M}_{k+1}$,~\eqref{eq:warm_start} stands for warm-start where $\tilde W_{k}$ is not updated and $\xi_k$ is used directly; and~\eqref{eq:restart} is restart where $\tilde W_{k+1}\gets \tilde W_k+\mathcal{M}_k(\xi_k)$ accumulating the historical information, and $\xi_{\text{new}}$ is the new initialization for the optimization within $\mathcal{M}_{k+1}$.
% \begin{equation}\label{eq:warm_start}
%     W_{k+1}=\tilde W_k+\mathcal{M}_{k+1}(\xi_k+\Delta \xi_k).
% \end{equation}
% \begin{equation}\label{eq:restart}
%     W_{k+1}=\tilde W_k+\mathcal{M}_k(\xi_k)+\mathcal{M}_{k+1}(\xi_{\text{new}}+\Delta \xi_k)=\tilde W_{k+1}+\mathcal{M}_{k+1}(\xi_{\text{new}}+\Delta \xi_k)
% \end{equation}
% A common choice of restart initialization is $\xi_{\text{new}}=0$, and we will introduce another important restart mechanism, called {\it full gradient restart}, in Section~\ref{sec:full_grad_restart} which is crucial for establishing the global convergence.

\subsection{Connection to Existing Benchmarks}\label{sec:connection}
%\cms{Consider providing a table to summarize these connections instead of writing wordy descriptions?}
While~\eqref{eq:representation_generic} may strike to be abstract, many existing benchmarks for LLM training can naturally fit in it by considering specific subspaces.
Here we summarize several notable methods in Table~\ref{tab:methods}.

% fixed-width, left-aligned paragraph columns
\newcolumntype{L}[1]{>{\raggedright\arraybackslash}p{#1}}
% flexible, left-aligned columns that expand
\newcolumntype{Y}{>{\raggedright\arraybackslash}X}

\begin{table*}[t]
\centering
\caption{Examples of memory-efficient training methods as instances of PESO.}
    \vskip -0.1in
\label{tab:methods}
\renewcommand{\arraystretch}{1.15}
\small
\begin{tabularx}{\linewidth}{L{2.5cm} L{0.8cm} L{1.0cm} Y p{3.0cm} >{\centering\arraybackslash}m{1.5cm}}
\toprule
\multicolumn{1}{l}{\textbf{Methods}} & 
\multicolumn{1}{l}{$\boldsymbol{\xi}$} & 
\multicolumn{1}{c}{$\boldsymbol{\mathcal{M}_k(\xi)}$} & 
\multicolumn{1}{c}{$\boldsymbol{\mathcal{S}_k}$} & 
\multicolumn{1}{c}{\textbf{\texttt{UpdateM}}} & 
\multicolumn{1}{c}{\textbf{Init.}} \\
\midrule
LoRA & $(A,B)$ &
$AB$ &
$\{A_kB + AB_k : A\!\in\!\mathbb{R}^{m\times r}, B\!\in\!\mathbb{R}^{r\times n}\}$ &
Adam for $A_k,B_k$ & \cellcolor{red!15}warm-start \\
\addlinespace
AdaLoRA & $\Lambda$ &
$P_k \Lambda Q_k$ &
$\{P_k \Lambda Q_k : \Lambda \in \mathbb{R}^{r\times r}\ \text{diagonal}\}$ &
SGD for $P_k, Q_k$ & \cellcolor{red!15}warm-start \\
\addlinespace
GaLore & $R$ &
$P_k R$ &
$\{P_k R : R \in \mathbb{R}^{r\times n}\}$ &
$P_k$: left $r$-SVD of $G_k$ & \cellcolor{blue!15}restart \\
\addlinespace
\cite{kozak2019stochastic} & $R$ &
$P_k R$ &
$\{P_k R : R \in \mathbb{R}^{r\times n}\}$ &
randomly sample $P_k$ & \cellcolor{blue!15}restart \\
\addlinespace
\cite{liang2024memory} & $R$ &
$P_k R$ &
$\{P_k R : R \in \mathbb{R}^{r\times n}\}$ &
online PCA of $P_k$ & \cellcolor{red!15}warm-start \\
\bottomrule
\end{tabularx}
    \vskip -0.18in
\end{table*}

\noindent \textbullet\ \underline{\textit{Projected subspace.}}
A simple way to define a memory-efficient subspace is through low-rank projection, where $\mathcal{M}_k:R\in\mathbb{R}^{r\times n}\mapsto P_kR$ is parameterized by a left-projection matrix $P_k\in\mathbb{R}^{m\times r}$. 
This formulation can be extended to right-sided or two-sided projections. 
By applying the chain rule to $\nabla_R\ell(\tilde W_k+P_kR)$, one obtains the projected subspace schemes analyzed in \citep{he2024subspace,kozak2019stochastic,liang2024memory,zhao2024galore}; see Appendix~\ref{sec:memory_eff_variant} for details. 
Within PESO, GaLore \citep{zhao2024galore}, APOLLO \citep{zhu2024apollo}, and stochastic subspace descent \citep{kozak2019stochastic} correspond to a \emph{restart} strategy by reassigning $P_k$, while online subspace descent \citep{liang2024memory} adopts a \emph{warm-start} update of $P_k$ via online PCA. 
%\cms{Correspondence between this paragraph to Table 1 is not clear. The last row of table 1 can be Projected subspace (liang et al. 2024), instead of a standalone citation.}
%(an example of choosing a different loss for {\tt UpdateMap} from the one for {\tt Opt}).

% In our framework, this corresponds to setting $\mathcal{M}_k(\xi_k) = P_k \xi_k$, where $\xi_k\in\mathbb{R}^{r\times n}$ is updated by $\Delta \xi_k=P_k^\top G_k$, and the full update is interpreted as a memory-efficient displacement within the subspace defined by $P_k$.
% GaLore \citep{zhao2024galore}, for example, updates $P_k$ periodically, which is a direct instance of our $K_1$-frequency subspace update for $\mathcal{M}_k$. 
% Similarly, \citep{liang2024memory} applies online PCA to update $P_k$, which corresponds to specifying a particular optimizer and loss function for {\tt OptM} in our framework. \cms{what is {\tt UpdateMap} in this case?}

\noindent \textbullet\ \underline{\textit{Low-rank subspace.}}
The LoRA family defines the subspace $\{A_kB + AB_k : A\!\in\!\mathbb{R}^{m\times r},\, B\!\in\!\mathbb{R}^{r\times n}\}$, where the adapters $(A,B)$ jointly serve as both $\xi$ and the parameterization of $\mathcal{M}_k$. 
Consequently, a single Adam update of $(A,B)$ simultaneously updates the subspace and its coordinates, effectively realizing a \emph{warm-start} scheme with $K=1$. 
Many LoRA variants can be viewed as modifications of this generic template: LoRA-Pro \citep{wang2024lora} applies a different preconditioner, and PiSSA \citep{meng2024pissa}, LoRA-GA \citep{wang2024loraga}, and LoRA-One \citep{zhang2025one} adjust initialization strategies, while other works modify learning rates or scaling rules. 
Our framework unifies these designs by interpreting them as \emph{specific choices of  \texttt{Opt} or initialization} within the same subspace structure.
% Vanilla LoRA corresponds to the special case of the representation~\eqref{eq:representation_SVD} where $\xi_k$ is fixed as a vector of ones throughout training, and $(U, V)$ are updated at every iteration using {\tt OptM}. 
% In terms of our framework (Algorithm~\ref{alg:full_details}), this corresponds to setting $K_1 = 1$, $K_2 = \infty$, and using a trivial {\tt Opt} that performs no updates.

% Various LoRA variants can also be interpreted within our framework by modifying components such as {\tt OptM}, hyperparameters, or the parameterization. 
% For instance, LoRA-GA and LoRA-Pro \citep{wang2024loraga,wang2024lora} specify different gradient-based update rules for {\tt OptM}; 
% LoRA+ and PiSSA \citep{hayou2024lora+,meng2024pissa} introduce custom initialization or learning rate schedules; 
% and AdaLoRA \citep{zhang2023adalora} allows $\xi$ to be learnable, aligning with our framework where {\tt Opt} is non-trivial. 
% LoRA-One \citep{zhang2025one} is equivalent to vanilla LoRA with a single full gradient restart applied at initialization.

\noindent \textbullet\ \underline{\textit{SVD subspace.}}
A principled way to extract low-dimensional structure from matrices (such as $\Delta W$) is through Singular Value Decomposition (SVD), leading to the representation $\mathcal{M}_k:\lambda\in\mathbb{R}^r \mapsto U \mathrm{diag}(\lambda) V$. 
Here, $(U,V)$ define the subspace (exploration), while $\lambda$ is the low-dimensional coordinates (exploitation). 
This separation fits directly into Algorithm~\ref{alg:generic_subspace}, enabling flexible optimization strategies for $(U,V)$ and $\lambda$. 
%\yl{[Rrepresentation cost and interpretability]}%\cms{In Section 3.2, you choose SVD parameterization over LoRA. It is good to discuss here some advantage of SVD subspace compared to LoRA.}
AdaLoRA \citep{zhang2023adalora} exemplifies this parameterization, and our framework clarifies the roles of $(P_k,\Lambda_k,Q_k)$ in their notation. 
We build on this principle in Section~\ref{sec:algorithms}, where we propose a practical SVD-based variant {\tt PESO-LoRA-T}.

{Together, these three subspace categories illustrate how PESO unifies existing PEFT methods under a single framework.
This perspective guides the design of practical algorithms with new exploration and exploitation techniques for PEFT, and it sets up the convergence analysis in the sections that follow.
}

\section{PESO-LoRA: Practical Algorithms from the Framework}\label{sec:algorithms}
The previous section establishes a conceptual link between PEFT methods and classical subspace minimization, providing a unifying interpretation. 
Building on this view, we now develop a concrete algorithm {\bf PESO-LoRA}, which combines PESO’s optimization principles with LoRA’s empirically effective PEFT parameterization.
We present two variants: 
{\tt PESO-LoRA-R} leverages a full gradient restart strategy to improve \emph{exploration} of subspaces, 
and {\tt PESO-LoRA-T} is a SVD-based method that enhances \emph{exploitation} through more effective optimization within each subspace.

\subsection{Full Gradient Restart}\label{sec:full_grad_restart}
We now introduce an important variant of the {\tt UpdateM} subroutine in the restart category (see~\eqref{eq:restart}) that enables convergence to stationarity in the full-parameter space.
We design {\tt UpdateM} so that each new subspace $\mathcal{S}_k$ induced by $\mathcal{M}_k$ remains well aligned with the current full gradient $G_k$. 
We call this scheme {\emph{full gradient restart}}:

\noindent\textbf{Full Gradient Restart.} Given learning rates $\{\eta_k\}$, whenever $k-1 \bmod K = 0$:
\begin{enumerate}[label=\arabic*), itemsep=0pt, topsep=2pt]
%\vspace{-0.8em}
    \item Absorb history: $\tilde W_{k} \gets \tilde W_{k-1} + \mathcal{M}_{k-1}(\xi_{k-1})$.
    \item Compute the (stochastic) full gradient $G_k = \nabla_W \ell(\tilde W_{k})$.
    \item Choose a low rank subspace $\mathcal{S}_k^{\text{FG}}$ depending on $G_k$. %\cms{this line is too verbal}
    \item Restart with $\xi_k\gets \xi_k^{\text{new}}$ such that $\mathcal{M}_k(\xi_k^{\text{new}})=-\eta_k P_{\mathcal{S}_k^{\text{FG}}}(G_k)$. %\cms{what is $\eta_k$?}
    %\vspace{-0.8em}
\end{enumerate}
Here, {$P_{\mathcal{S}_k^{\text{FG}}}(G_k)$ denotes the projection of $G_k$ onto $\mathcal{S}_k^{\text{FG}}$}. 
This procedure effectively redefines $\mathcal{M}_k$ so that the new adapter is initialized by a {\it projected gradient step}:
\begin{equation}\label{eq:grad_restart_generic}
 \textstyle   W_{k} \gets W_{k-1} - \eta_k P_{\mathcal{S}_k^{\text{FG}}}(G_k).
\end{equation}
Thus, each restart ensures that $\mathcal{S}_k$ captures information from full gradients, with initial progress comparable to a standard SGD step.
In the literature on subspace methods, incorporating the full gradient into $\{\mathcal{S}_k\}$ is critical for convergence guarantees \citep{conn1994iterated,zhang2025scalable}. 
In particular, one can show that $\|\nabla_W \ell\| \to 0$ provided that $G_k := \nabla_W \ell(W_k)$ lies in $\mathcal{S}_k$. 
Building on this, we demonstrate in Section~\ref{sec:analysis} that full gradient restart ensures convergence to a stationary point of the original problem~\eqref{opt_original} by interleaving projected steepest-descent steps with subspace updates.
% As long as each subspace optimization yields sufficient decrease in the loss, the algorithm remains close to the classical trajectory of (stochastic) gradient descent. 
% We provide a detailed convergence analysis in Section~\ref{sec:analysis}.

%\vspace{-0.5em}
\begin{algorithm}[ht]
\caption{{\tt PESO-LoRA-R}: {\bf PESO} with {\bf LoRA} and Subspace Explo{\bf R}ation}
\label{alg:lora_restart}
\begin{algorithmic}[1]
    \Require Pre-trained parameters $W_0\in\mathbb{R}^{m\times n}$, frequency $K$, scale parameter $\gamma$.
    \State Set $k \gets 1$, $\tilde W_0\gets W_0$, $A_0\gets 0$ and $B_0\gets 0$.
    \While{stopping criteria not satisfied}
    \If{$k-1\mod{K}=0$}
    \State $\tilde W_k\gets \tilde W_{k-1}+A_{k-1}B_{k-1}$.
    \State Compute stochastic full gradient $G_k$.
    \State $(U_k,\Lambda_k,V_k)\gets {\tt SVD}(-G_k)$. %\Comment{Top-$r$ SVD of $G_k$}
    \State $A_{k-1}\gets \frac{1}{\sqrt{\gamma}} U_k\sqrt{\Lambda_k}$,  $B_{k-1}\gets  \frac{1}{\sqrt{\gamma}} \sqrt{\Lambda_k}V_k$.
    \EndIf %\zy{Shall we mention \(\gamma\) here or in the Appendix?}\yl{What to say?}
    \State $(A_k, B_k)\gets \texttt{AdamW}(A_{k-1},B_{k-1})$. %{\Comment{One AdamW step on $(A_{k-1},B_{k-1})$}}  %\cms{is this a standard representation of saying apply Adam to update $A$ and $B$? More precisely, this is only a step of Adam.}
    %\State $(A_k, B_k)\gets(A_{k-1},B_{k-1})+(\Delta A_k, \Delta B_k)$.
    \State $k \gets k+1$.
    \EndWhile 
    \State \Return $\tilde W_k+A_kB_k$.
\end{algorithmic}
\end{algorithm}
%\vspace{-0.5em}

A practical construction of $\mathcal{S}_k^{\text{FG}}$ is to compute a rank-$r$ SVD of $G_k$ and define the subspace as the span of its top singular directions. 
This ensures that $\mathcal{S}_k^{\text{FG}}$ captures the main structure of $G_k$, while the approximation error $\|G_k - P_{\mathcal{S}_k^{\text{FG}}}(G_k)\|$ is governed by the spectral tail of $G_k$. 
Crucially, this tail is {\it independent} of the rank gap in the objective, underscoring a key distinction between representation deficiency (e.g., LoRA) and update efficiency. 
In practice, {given $(U_k,V_k)$ from the rank-$r$ SVD of $G_k$}, one can also restart with $\mathcal{M}_k(\xi_k^{\text{new}})=-\eta_k U_k V_k$, which corresponds to the update from the recent training benchmark Muon \citep{jordan2024muon}, providing improved stability over~\eqref{eq:grad_restart_generic}.

One practical advantage of full gradient restart is that it acts as a {\bf ``plug-and-play''} mechanism for existing PEFT methods. 
It can be applied with a moderate frequency $K$ to reduce the cost of SVD while still guiding subspace exploration effectively. 
Recent work, such as \citep{zhang2025one,zhao2024galore}, has demonstrated the empirical benefits of leveraging the full gradient.
In particular, the recent variants LoRA-GA \citep{wang2024loraga} and LoRA-One \citep{zhang2025one} can be interpreted as special cases of LoRA with full gradient restart applied {\it only at initialization}.
To achieve the full-parameter convergence, we propose {\tt PESO-LoRA-R} (Algorithm~\ref{alg:lora_restart}), which embeds LoRA with a restart mechanism every $K$ iterations.
Here, {\tt SVD}$(-G_k)$ returns the top-$r$ SVD of $-G_k$.
{A detailed version of the pseudocode can be found in Appendix~\ref{app:algo_details}.}

To implement Algorithm~\ref{alg:lora_restart}, directly assigning $(A_{k-1},B_{k-1})\gets 1/{\sqrt{\gamma}} (U_k\sqrt{\Lambda_k},\sqrt{\Lambda_k}V_k)$ can cause instability due to mismatches in optimization states.
{A similar instability issue has also been reported for GaLore \citep{chen2024fira}.
For robustness, we propose alignment techniques to maintain consistency of subspace bases, momentum, and velocity.
In particular, restarts often produce gradients with much larger magnitudes, leaving the Adam velocity ``too cold'' and causing unstable steps. We therefore apply a lightweight {\it velocity alignment}, rescaling the velocity to match the new gradient magnitude, $v \gets \|g\|^2/\|v\|\,v$, together with a short $\beta_2$ warm-up, which we find to be the most important component for stable performance in practice; see more details in Appendix~\ref{sec:implementation}.}
Finally, we remark both empirical evidence and theoretical results suggest that gradients $G_k$ in deep learning often have strong low-rank structure, making them especially suitable for efficient SVD-based approximations \citep{cosson2023low,yang2023spectral,zhao2024galore}.

% To implement Algorithm~\ref{alg:lora_restart}, directly assigning $(A_{k-1},B_{k-1})\gets  1/{\sqrt{\gamma}} (U_k\sqrt{\Lambda_k},\sqrt{\Lambda_k}V_k)$ can cause instability due to mismatches in optimization states.
% {This issue also happens for GaLore, as discussed in \citep{chen2024fira}.}
% For robustness, we propose alignment techniques to maintain consistency of subspace bases, momentum, and velocity. 
% More details are provided in Appendix~\ref{sec:implementation}.
% Finally, we remark both empirical evidence and theoretical results suggest that gradients $G_k$ in deep learning often have strong low-rank structure, making them especially suitable for efficient SVD-based approximations \citep{cosson2023low,yang2023spectral,zhao2024galore}.

\subsection{Exploitation via SVD Subspace}\label{sec:exploit_algo}
Having discussed exploration techniques inspired by subspace minimization, we now turn to the complementary philosophy: {exploitation within the current subspace}.

%\vspace{-0.5em}
\begin{algorithm}[ht]
\caption{{\tt PESO-LoRA-T}: {\bf PESO} with {\bf LoRA} and Subspace Exploi{\bf T}ation}
\label{alg:SVD_exploit}
\begin{algorithmic}[1]
    \Require Pretrained weights $W_0\in\mathbb{R}^{m\times n}$, initial subspace matrices $U_0\in\mathbb{R}^{m\times r}$, $V_0\in\mathbb{R}^{r\times n}$, initial coordinate $\xi_0\in\mathbb{R}^r$, frequency $K$.
    \State Set $k \gets 1$.
    \While{stopping criterion not met}
        \State Keep $(U_{k},V_{k}) \gets (U_{k-1},V_{k-1})$.
        \If{$k-1 \bmod K = 0$}
            \State $(U_{k},V_{k})\gets \texttt{AdamW}(U_{k-1},V_{k-1})$. %\Comment{One AdamW step on $(U_{k-1},V_{k-1})$}
        \EndIf
        \State $\xi_{k}\gets \texttt{AdamW}(\xi_{k-1})$. %\Comment{One AdamW step on $\xi_{k-1}$}
        \State $k \gets k+1$.
    \EndWhile
    \State \Return $W_0 + U_k \operatorname{diag}(\xi_k) V_k$.
\end{algorithmic}
\end{algorithm}
%\vspace{-0.5em}

As outlined in Section~\ref{sec:connection}, an SVD-based parameterization provides a clean and principled way to define $\mathcal{M}_k$. 
Specifically, we approximate the target mapping $\mathcal{M}(\xi^*)$ by a sum of rank-$r$ components, $\sum_k U_k \xi^*_k V_k$, where each pair $(U_k,V_k)$ defines an {\it SVD subspace} of rank $r$. 
Because SVD naturally captures the dominant gradient directions, this parameterization ensures that exploitation is focused on the most informative directions in the weight space. 

Within each subspace, we optimize the low-dimensional $\xi$ for $K$ steps using AdamW. 
This design can be viewed as an extension of LoRA, with the key difference that the SVD structure explicitly decouples subspace exploitation (through $\xi$) from exploration (through $(U,V)$). 
The practical variant is summarized in Algorithm~\ref{alg:SVD_exploit}. A small frequency $K$ (e.g., $1$ or $2$) often suffices for strong performance without significant overhead.
{A detailed version of Algorithm~\ref{alg:SVD_exploit} can be found in Appendix~\ref{app:algo_details}.}

%Moreover, this setup highlights the exploration–exploitation balance: while $(U,V)$ are updated infrequently to guide exploration of new subspaces, $\xi$ is updated continuously to fully exploit the current subspace. 
%We report numerical results in Section~\ref{sec:experiments} and defer additional implementation details to Appendix~\ref{sec:implementation}.

\section{Experiments}\label{sec:experiments} 
In this section, we conduct experiments to evaluate our methods across diverse tasks and models, comparing with standard LoRA-based approaches and full fine-tuning. 
We first assess natural language understanding on the GLUE benchmark~\citep{wang2018glue} by fine-tuning T5-base~\citep{raffel2020exploring}. 
We then evaluate natural language generation on Llama models for tasks including mathematical reasoning, code generation, and general instruction tuning. 
Finally, we demonstrate that \(\texttt{LoRA-PESO-R}\) remains effective even under strict memory constraints when trained for more epochs. Implementation details are provided in Appendix~\ref{app:experiments}. 
%We also note that the additional memory cost of \(\texttt{PESO}\) compared to standard LoRA-based methods is negligible.

\begin{table*}[ht]
\centering
%\vspace{-0.3em}
\caption{Fine-tuned T5-base on natural language understanding tasks with rank set to 8. Results are reported as accuracy (\%) over 3 runs. \textbf{Bold} and \underline{underline} indicate the highest and second-highest accuracies \emph{excluding} {\tt PESO-LoRA-T}, which is shaded in gray and omitted from direct comparison due to its longer runtime.
}
\label{tab:result_nlu}
\begin{tabular}{lccccc}
\toprule
Method & MNLI & SST-2 & CoLA & QNLI & MRPC  \\
\midrule
LoRA       & 85.30$_{\pm0.04}$ & 94.04$_{\pm0.09}$ & 72.84$_{\pm1.25}$ & 93.02$_{\pm0.07}$ & 68.38$_{\pm0.01}$  \\
LoRA+      & 85.81$_{\pm0.09}$ & 93.85$_{\pm0.24}$ & 77.53$_{\pm0.20}$ & 93.14$_{\pm0.03}$ & 74.43$_{\pm1.39}$  \\
P-LoRA     & 85.28$_{\pm0.15}$ & 93.88$_{\pm0.11}$ & 79.58$_{\pm0.67}$ & 93.00$_{\pm0.07}$ & 83.91$_{\pm1.16}$  \\
PiSSA      & 85.75$_{\pm0.07}$ & 94.07$_{\pm0.06}$ & 74.27$_{\pm0.39}$ & 93.15$_{\pm0.14}$ & 76.31$_{\pm0.51}$ \\
LoRA-GA    & 85.70$_{\pm0.09}$ & 94.11$_{\pm0.18}$ & 80.57$_{\pm0.20}$ & 93.18$_{\pm0.06}$ & 85.29$_{\pm0.24}$ \\
LoRA-Pro   & \underline{86.03}$_{\pm0.19}$ & 94.19$_{\pm0.13}$ & \underline{81.94}$_{\pm0.24}$ & \underline{93.42}$_{\pm0.05}$ & \underline{86.60}$_{\pm0.14}$ \\
LoRA-One   & 85.89$_{\pm0.08}$ & \underline{94.53}$_{\pm0.13}$ & \textbf{82.04}$_{\pm0.22}$ & 93.37$_{\pm0.02}$ & \textbf{87.83}$_{\pm0.37}$ \\
\midrule
{\tt PESO-LoRA-R} & \textbf{86.08}$_{\pm0.15}$ & \textbf{94.61}$_{\pm0.09}$ & 81.50$_{\pm0.16}$ & \textbf{93.43}$_{\pm0.06}$ & 86.36$_{\pm0.11}$\\
\rowcolor{rowgray} % or: \rowcolor{gray!10}
{\tt PESO-LoRA-T} & 86.09$_{\pm0.04}$ & 94.76$_{\pm0.19}$ & 82.01$_{\pm0.30}$ & 93.45$_{\pm0.03}$ & 87.59$_{\pm0.46}$\\
\bottomrule
\end{tabular}
%\vspace{-1em}
\end{table*}
\subsection{Natural Language Understanding Tasks}\label{subsec:nlu_tasks}
We fine-tune the T5-base model on a subset of GLUE, including MNLI, SST-2, CoLA, QNLI, and MRPC, and evaluate performance using test accuracy (\%). Following the setting in~\citep{zhang2025one}, we compare our method against several LoRA variants, including LoRA~\citep{hu2022lora}, LoRA+~\citep{hayou2024lora+}, P-LoRA~\citep{zhang2024riemannian}, PiSSA~\citep{meng2024pissa}, LoRA-GA~\citep{wang2024loraga}, LoRA-Pro~\citep{wang2024lora}, and LoRA-One~\citep{zhang2025one}. For fairness, hyperparameters are tuned individually for each method. 

% Additional experimental details can be found in Appendix~\ref{app:nlu_experiments}.
% The results are presented in Table~\ref{tab:result_nlu}. We observe that {\tt PESO-LoRA-R} and {\tt PESO-LoRA-T} achieve the best performance on three of the five GLUE tasks (MNLI, SST-2, and QNLI), which are also the relatively larger datasets. 
% % On the remaining two tasks, although not the best-performing methods, they still deliver competitive results. 
% On the remaining two tasks, {\tt PESO-LoRA-T} still achieves the second best. 
% This highlights both the robustness and strong overall performance of our approaches, with their advantages becoming more pronounced on larger datasets that require longer training and stronger exploration–exploitation capabilities. 
% Moreover, {\tt PESO-LoRA-T} generally outperforms {\tt PESO-LoRA-R}, but at the cost of increased computation time—approximately 
% \(1.4 \times\) that of other LoRA variants—whereas {\tt PESO-LoRA-R} runs at nearly the same speed as standard LoRA methods. 
% The memory costs are roughly the same among all methods.
% Thus, the choice between the two depends on the user’s priorities in balancing performance and efficiency.
The results are summarized in Table~\ref{tab:result_nlu}. 
{\tt PESO-LoRA-R} and {\tt PESO-LoRA-T} achieve the best performance on three of the five tasks, which also have  {\it larger datasets}. 
On the remaining tasks, {\tt PESO-LoRA-T} ranks second. 
This demonstrates the overall efficiency and robustness of our approaches, with advantages most evident on larger datasets that demand longer training and stronger exploration–exploitation.
Moreover, {\tt PESO-LoRA-T} generally outperforms {\tt PESO-LoRA-R}, but at the cost of \(1.4\times\) more computation time, whereas {\tt PESO-LoRA-R} runs at nearly the same speed as standard LoRA. 
Memory costs are comparable across all methods, so the choice ultimately depends on whether performance or efficiency is prioritized.

% In your preamble:
% \usepackage[table]{xcolor}
% \definecolor{rowgray}{gray}{0.94} % optional; or use gray!10 inline

\subsection{Natural Language Generation Tasks}
\label{sec:nlg_tasks}

We evaluate our method on the more capable Llama-3.1-8B~\citep{dubey2024llama} model across three downstream tasks. For mathematical reasoning, we fine-tune on a 100k subset of MetaMathQA~\citep{yu2023metamath} and evaluate on GSM8K~\citep{cobbe2021training}. For general instruction tuning, we fine-tune on Alpaca~\citep{alpaca} and evaluate on MMLU~\citep{hendrycks2020measuring}. For code generation, we use a 100k subset of Code-Feedback~\citep{zheng2024opencodeinterpreter} and evaluate on HumanEval~\citep{chen2021evaluating}, reporting PASS@1. To ensure fairness, all datasets are preprocessed to exclude overlaps with test sets. 

The results are reported in Table~\ref{tab:result_nlg_llama3}. \texttt{PESO-LoRA-R} consistently outperforms all LoRA variants across all benchmarks on this advanced model. 
Notably, the gains are most pronounced in mathematical reasoning and code generation—tasks which involves larger fine-tuning datasets, highlighting the substantial advantages of subspace exploration and exploitation for complex tasks. 
For completeness, Table~\ref{tab:result_nlg} in Appendix~\ref{app:nlg_experiments} reports results on Llama-2-7B~\citep{touvron2023llama}, evaluating both \texttt{PESO-LoRA-T} and \texttt{PESO-LoRA-R}, where our methods similarly achieve competitive performance against established baselines. We further compare against pretraining-oriented approaches, including GaLore~\citep{zhao2024galore}, Fira~\citep{chen2024fira}, and APOLLO~\citep{zhu2024apollo}; Table~\ref{tab:subspace_comparison} shows that our PEFT-specialized \texttt{PESO-LoRA-R} consistently outperforms these pretraining baselines in the fine-tuning regime.

We conduct ablations over the rank $r$ and restart frequency $K$. Results in Appendix~\ref{app:nlg_experiments_ablation} show that \texttt{PESO-LoRA-R} achieves strong performance with small ranks and remains robust across a wide range of $K$. Appendix~\ref{app:nlg_experiments_ablation} also reports memory and runtime comparisons, showing a memory footprint similar to vanilla LoRA and comparable overall computational cost. In particular, Table~\ref{tab:ablation_cost} shows that the additional SVD cost over LoRA is only 1.9\% of total runtime in our experiments, indicating negligible overhead.

\begin{table}[htbp]
\centering
%\vspace{-0.5em}
\caption{Fine-tuned Llama-3.1-8B on natural language generation tasks with rank set to 8. Results are reported as accuracy (\%) over 3 runs. \textbf{Bold} and \underline{underline} indicate the highest and second-highest accuracies.}
%\vspace{-0.2em}
\label{tab:result_nlg_llama3}
{
    \begin{tabular}{c|cccc}
    \toprule
     & LoRA & LoRA-GA & LoRA-One & \texttt{PESO-LoRA-R} \\
    \midrule
    GSM8K & 70.64$_{\pm0.53}$ & 76.67$_{\pm0.31}$ & \underline{77.71}$_{\pm0.17}$ & \textbf{77.79}$_{\pm0.18}$ \\
    MMLU & 63.95$_{\pm0.05}$ & 62.91$_{\pm0.08}$ & \underline{64.33}$_{\pm0.14}$ & \textbf{64.34}$_{\pm0.21}$ \\
    HumanEval & 42.47$_{\pm2.56}$ & 44.32$_{\pm5.64}$ & \underline{45.32}$_{\pm1.52}$ & \textbf{47.15}$_{\pm0.76}$ \\
    \bottomrule
    \end{tabular}
}
%\vspace{-0.5em}
\end{table}

\begin{table}[htbp]
\centering
%\vspace{-0.5em}
\caption{Fine-tuned T5-base (4 epochs) on the SST-2 dataset. Results are reported as accuracy (\%) over 3 runs. \textbf{Bold} and \underline{underline} indicate the highest and second-highest accuracies.}
\label{tab:epoch_results}
%\vspace{-0.2em}
{
    \begin{tabular}{c|cccc}
    \toprule
    Method & Epoch 1 & Epoch 2 & Epoch 3 & Epoch 4 \\
    \midrule
    LoRA ($r=2$) & $93.02_{\pm0.44}$ & $93.47_{\pm0.51}$ & $93.41_{\pm0.12}$ & $93.52_{\pm0.17}$ \\
    LoRA ($r=4$) & $94.23_{\pm0.30}$ & $94.42_{\pm0.15}$ & $94.46_{\pm0.10}$ & $94.61_{\pm0.25}$ \\
    LoRA ($r=8$) & $93.85_{\pm0.30}$ & $94.03_{\pm0.09}$ & $94.54_{\pm0.05}$ & $94.54_{\pm0.23}$ \\
    \texttt{PESO-LoRA-R} ($r=2$) & $\underline{94.30}_{\pm0.15}$ & $\underline{94.47}_{\pm0.08}$ & $\underline{94.84}_{\pm0.25}$ & $\textbf{95.14}_{\pm0.15}$ \\
    Full fine-tuning & $\textbf{94.42}_{\pm0.11}$ & $\textbf{94.70}_{\pm0.10}$ & $\textbf{94.85}_{\pm0.11}$ & $\underline{94.90}_{\pm0.06}$ \\
    \bottomrule
    \end{tabular}
}
%\vspace{-0.5em}
\end{table}

\begin{table}[htbp]
\centering
%\vspace{-0.5em}
\caption{Comparison with pretraining-oriented subspace methods. Fine-tuned Llama-3-8B-Instruct (Alpaca-en-demo, $r=8$). Results are MMLU accuracy (\%). Baseline results from \citet{zhu2024apollo}. \textbf{Bold} and \underline{underline} indicate the highest and second-highest accuracies.
}
%\vspace{-0.2em}
\label{tab:subspace_comparison}
{
    \begin{tabular}{c|ccccc}
    \toprule
    Method & STEM & Social Sciences & Humanities & Others & Average \\
    \midrule
    GaLore & 54.50 & 75.11 & 58.59 & 72.03 & 64.43 \\
    Fira & 53.53 & 75.46 & 58.59 & 72.09 & 64.32 \\
    APOLLO /w SVD & \underline{54.73} & 75.46 & \underline{58.72} & \underline{72.68} & \underline{64.76} \\
    APOLLO & 54.37 & \underline{75.86} & 58.18 & 71.69 & 64.35 \\
    \midrule
    \texttt{PESO-LoRA-R} & \textbf{57.15} & \textbf{76.86} & \textbf{60.79} & \textbf{73.57} & \textbf{66.47} \\
    \bottomrule
    \end{tabular}
    %\vspace{-0.5em}
}
\end{table}

\subsection{Multi-Epoch Low-Rank Analysis}
\label{sec:long_epoch}
To highlight the effect of subspace exploration, we extend T5-base fine-tuning on SST-2 from one epoch (Section~\ref{subsec:nlu_tasks}) to four. 
The longer schedule allows more thorough exploration and mitigates the intrinsic low-rank bottleneck. 
As shown in Table~\ref{tab:epoch_results}, \texttt{PESO-LoRA-R} with $r=2$ consistently outperforms standard LoRA even with larger ranks ($r=4,8$), indicating that it alleviates the low-rank limitation and achieves stronger performance under tightly restricted memory budgets.

% To demonstrate the effectiveness of subspace exploration, we extend the fine-tuning time of the T5-base model on SST-2 from one epoch (as used in Section~\ref{subsec:nlu_tasks}) to four epochs. This longer training schedule allows more thorough subspace exploration, thereby diminishing the intrinsic low-rank bottleneck. The results are reported in Table~\ref{tab:epoch_results}. We find that \texttt{PESO-LoRA-R} with $r=2$ consistently outperforms standard LoRA even at higher ranks as high as $r=8$, demonstrating that it successfully overcomes the intrinsic low-rank limitation to converge to full-parameter optimality and enables stronger performance with better memory efficiency.

\section{Convergence Analysis}\label{sec:analysis}
\begin{comment}
In this section, we establish convergence of PESO with the full gradient restart. For general nonconvex losses, the optimality measure $\mathbb{E}\|G_k\|$ converges to zero up to controlled inexactness. We begin by stating the regularity assumptions.
% We begin by recalling some notation. 
% Let $\mathcal{S}_k$ denote the subspace constructed at iteration $k$, and let $G_k$ represent the full gradient with respect to $W_k$. 
% For any set $X$, we define $P_X(x)$ as the projection of $x$ onto $X$.

\begin{assumption}\label{assump:loss_regular}
    The loss $\ell$ is nonconvex, bounded from below, and has $L$-Lipschitz gradients.
\end{assumption}
\begin{assumption}\label{assump:moments}
    Stochastic gradients $\tilde G_k$ of the full gradient $G_k$ satisfy $\mathbb{E}(\tilde G_k)=G_k$ and there exists $C>0$ such that $\mathbb{V}(\tilde G_k)\leq C$.
\end{assumption}
\end{comment}

In this section we establish convergence guarantees for \texttt{PESO-LoRA-R}, a practical instantiation of PESO.
A more general asymptotic convergence result for PESO is provided in Appendix~\ref{sec:asym}.
We consider a nonconvex objective $\ell(W)$ and make the following standard assumptions.

\begin{assumption}
\label{assump:loss_regular}
The loss $\ell$ is bounded below and has $L$-Lipschitz continuous gradients.
\end{assumption}

\begin{assumption}
\label{assump:moments}
Stochastic gradients $\tilde G_k$ of the full gradient $G_k=\nabla\ell(W_k)$ are accessible and satisfy $\mathbb{E}(\tilde G_k)=G_k$.
There exists $C>0$ such that $\mathbb{V}(\tilde G_k)\leq C$.
%Let $\tilde G_k$ be an unbiased estimator of $G_k=\nabla \ell(W_k)$ such that $\mathbb{E}\|\tilde G_k-G_k\|^2 \le C$.
%\vspace{-0.5em}
\end{assumption}

We analyze \texttt{PESO-LoRA-R} (Algorithm~\ref{alg:lora_restart}), where the inner optimizer is set as Adam.
Below we state a standard regularity assumption on the low-rank objective $f(A,B)=\ell(W+AB)$ optimized by Adam.

\begin{assumption}
\label{assump:inner}
At each iteration, \texttt{PESO-LoRA-R} applies Adam to the low-rank objective
$f(A,B)=\ell(W+AB)$.
We assume for $f$:
(i) $\nabla f$ is Lipschitz continuous,
(ii) stochastic gradients are unbiased, and
(iii) stochastic gradient variance is bounded.
\end{assumption}

Let $P_r(G_k)$ denote the rank-$r$ truncated SVD of $G_k$.
Since $P_r(\cdot)$ is nonlinear, in general
$\mathbb{E}[P_r(\tilde G_k)] \neq P_r(\mathbb{E}[\tilde G_k])$.
Our analysis only requires that this projection-induced bias is asymptotically negligible in a stepsize-weighted sense.

\begin{assumption}
\label{assump:subspace}
At each restart step $(k-1)\bmod K=0$, the estimated subspace satisfies $\|\mathbb{E}[ P_r(\tilde G_k)] - P_r(G_k)\|_F \le \varepsilon_k,
\sum_k \eta_k \varepsilon_k^2 < \infty.$
\end{assumption}
\begin{rem}
This assumption holds in several standard regimes in subspace optimization literature. 
It holds trivially in the {full-batch} setting ($\varepsilon_k=0$) and is satisfied in the {large-batch} regime via concentration bounds $\varepsilon_k=\mathcal{O}(|\mathcal{B}|^{-1/2})$; see, e.g., \citep{he2024subspace}. 
Furthermore, using {independent randomized subspaces} \citep{chen2025memory} yields $\varepsilon_k=0$ due to conditional unbiasedness.
% Assumption~\ref{assump:subspace} in fact can recall several assumptions used in the subspace optimization literature to achieve a similar rate of convergence.
% In particular:
% (i) in the full-batch case, $\tilde G_{k_s}=G_{k_s}$ so $\varepsilon_s=0$;
% (ii) for large-batch gradients, concentration typically yields $\varepsilon_s=\mathcal{O}(|\mathcal{B}|^{-1/2})$ under mild regularity (see a large batch convergence result of GaLore in \citep{he2024subspace});
% (iii) if instead of SVD for randomized subspaces $\mathcal{S}_k$ independent of $\tilde G_{k_s}$ as in \citep{chen2025memory,he2024subspace,malinovsky2024randomized}, one can often obtain $\varepsilon_s=0$ due to $\mathbb{E}[P_{\mathcal{S}_k}(\tilde G_k)]=P_{\mathcal{S}_k}(\mathbb{E}[\tilde G_k])$.
\end{rem}

\begin{thm}
\label{thm:peso_rate}
Under Assumptions~\ref{assump:loss_regular}--\ref{assump:subspace} and following the hyperparameter choices  from Appendix~\ref{sec:proof_details},
the iterates generated by \texttt{PESO-LoRA-R} satisfy
\[
\textstyle
\min_{0 \le k \le T-1}
\mathbb{E}\|\nabla \ell(W_k)\|_2^2
= \mathcal{O}\!\left(1/{\sqrt{T}}\right).
\]
\end{thm}

% \paragraph{Discussion.}
% Theorem~\ref{thm:peso_rate} establishes a non-asymptotic convergence rate for
% \texttt{PESO-LoRA-R} in the full parameter space, despite optimization being performed in low-rank subspaces.
% The result highlights the role of periodic subspace refreshment in preventing stagnation.
% A more general asymptotic convergence result for PESO is deferred to
% Appendix~\ref{sec:proof_details}.

\section{Conclusions and Limitations}
This paper bridges classical methodology from nonlinear optimization with the practical challenge of memory-efficient LLM fine-tuning. 
It highlights two key perspectives: 1) practical constraints in LLM training, such as memory limits, can motivate specialized optimization designs; 2) principles from nonlinear optimization can in turn guide the development of practical algorithms for LLMs. 
We believe this opens promising directions for principled and scalable LLM training, while underscoring a broader philosophy: the rapid progress in LLMs can be enriched by classical foundations in computation and optimization.
Our study has certain limitations. 
Due to limited resources, our experiments are restricted to medium-scale settings and do not yet reach the largest practical regimes. 
Extending our framework to full-scale pre-training remains an important future work, and we expect the methodology developed here to provide a solid foundation for such efforts.

\section*{Acknowledgments}
We would like to thank Professor Jorge Nocedal for his invaluable insights and discussions, which greatly assisted in formulating the framework and refining the theoretical development of this work.

%\clearpage
% \section*{Impact Statement}
% This paper contributes to the advancement of Machine Learning. While it may have broader societal impacts, we do not identify any that warrant specific discussion.
% \section*{Ethics statement}
% Our research fully adheres to the ICLR Code of Ethics. 
% We did not use any human or animal subjects. All datasets and models were acquired and used in accordance with their respective usage guidelines, and no private data was compromised. 
% This work is free from bias and discriminatory outcomes, avoids using personally identifiable information, and presents no risks to privacy or security. 
% We are committed to conducting this research with complete transparency and integrity.

% \section*{Reproducibility Statement}
% We take reproducibility seriously and are willing to provide all necessary materials to support it. 
% All theoretical results are presented with explicit assumptions, and full proofs are provided in Appendix~\ref{sec:proof_details}.
% Additionally, experimental settings and implementation details are documented in Appendix~\ref{sec:synthetic}, \ref{sec:implementation} and \ref{app:experiments}. 
% Together, these resources allow our claims and results to be verified and reproduced.

\bibliographystyle{icml2026}
\bibliography{references}

\appendix
\onecolumn
\section{Synthetic Example of LoRA Deficiency}\label{sec:synthetic}
One critical limitation in the literature is the absence of convergence guarantees toward valid optimality conditions of \eqref{opt_original}.
Most existing works establish convergence only with respect to the low-dimensional parameters—such as the factors $A$ and $B$ in LoRA—but do not address convergence with respect to the full-parameter matrix $W$.
For instance, \citet{jiang2024unified} shows that $\nabla_A \ell(W_0 + A_k B_k)$ and $\nabla_B \ell(W_0 + A_k B_k)$ vanish as $k \to \infty$, while leaving the behavior of $\nabla_W \ell$ uncharacterized.
This gap is not merely technical: it highlights a fundamental deficiency of PEFT methods compared to standard full-parameter training. 
In fact, the optimal loss attained by LoRA can be arbitrarily worse than the true optimal loss of \eqref{opt_original}.
To illustrate this, consider the following simple synthetic example in matrix optimization:
\begin{equation}\label{opt:example} 
\min_{W\in\mathbb{R}^{n\times n}}\quad \|W-M\|_F^2,\quad\text{ where }M = a \cdot \operatorname{diag}(1,\dots,1,0,\dots,0), \quad (r{+}1 \text{ ones}).
\end{equation}
The optimal solution is clearly $W^* = M$ with $f(W^*) = 0$.
However, applying LoRA with rank $r$ to \eqref{opt:example} can at most achieve a rank-$r$ approximation of $M$, and attains $f(A^*B^*) = a^2$.
As $a$ increases, or as the rank mismatch between the LoRA adapters and the true solution grows, the optimality gap between LoRA and full gradient methods can become arbitrarily large.

This example underscores the cost of memory restrictions: while low-rank parameterizations save memory, they may fundamentally limit convergence to the true optimum.
In the middle and right panels of Figure~\ref{fig:examples}, we show that LoRA with exploration ({\tt PESO-LoRA-R}) can effectively converge to the true optimal solution while LoRA would not.
Note in this implementation of {\tt PESO-LoRA-R} (Algorithm~\ref{alg:lora_restart}), the SVD of the full gradient is a rank-$r$ SVD and therefore the low-rankness of this computational scheme would not affect the convergence.

\section{Review on Subspace Minimization}\label{sec:subspace_review}
It is worth noting that Algorithm~\ref{alg:generic_subspace} is essentially equivalent to the classical two-loop subspace minimization scheme \citep{conn1994iterated}, summarized in Algorithm~\ref{alg:subspace_algo}. 

The key distinction between warm-start and restart, discussed in~\eqref{eq:warm_start} and~\eqref{eq:restart} of Section~\ref{sec:explore_exploit}, lies in how the subproblem~\eqref{eq:subspace_opt} is initialized within each outer iteration of Algorithm~\ref{alg:subspace_algo}.

\begin{algorithm}[tb]
\caption{Classical Iterated-Subspace Minimization}
\label{alg:subspace_algo}
\begin{algorithmic}[1]
    \Require Initialization $W_0\in\mathbb{R}^{m\times n}$, $\xi_0\in\mathbb{R}^d$, $\mathcal{M}_0$, an algorithmic subroutine \texttt{UpdateM}, an optimizer \texttt{Opt}, frequency $K$.
    \State Set $k \gets 1$ and $\tilde W_0\gets W_0$.
    \While{stopping criteria not satisfied}
    \State Update the subspace $(\mathcal{M}_{k},\tilde W_{k})\gets \texttt{UpdateM}(\mathcal{M}_{k-1},\tilde W_{k-1})$.
    \State Approximately solve the subspace minimization by {\tt Opt} using $K$ inner-loop iterations:
    \begin{equation}\label{eq:subspace_opt}
        \xi_k^*\gets \text{approx}\arg\min_{\xi}\,\,\ell(\tilde W_{k}+\mathcal{M}_{k}(\xi))
    \end{equation}
    % \State $\xi_{k+1}\gets \xi_{k}+\Delta \xi_k^*$.
    \State $W_{k}\gets \tilde W_{k}+\mathcal{M}_{k}(\xi_k^*)$.
    \State $k \gets k+1$.
    \EndWhile
\end{algorithmic}
\end{algorithm}

\section{Projected Subspace and Memory Efficiency}\label{sec:memory_eff_variant}

One-sided projected subspaces in PESO can offer stronger memory efficiency than LoRA.  
This idea is exemplified by GaLore \citep{zhao2024galore}, which we now place in the PESO framework.  
GaLore requires memory of order $mn + mr + 2nr$ (assuming $m \le n$), compared to LoRA’s $mn + 3mr + 3nr$.  

With the projected subspace representation in Table~\ref{tab:methods}, optimization reduces to $\xi := R \in \mathbb{R}^{r \times n}$.  
By the chain rule,
\begin{equation}
    \nabla_{R}\ell(W_k) 
    = \nabla_{\xi}\ell(\tilde{W}_k + P_k R_k) 
    = P_k^\top \nabla_W \ell(W_k) 
    = P_k^\top G_k.
\end{equation}
Once $P_k$ is computed and stored, subspace gradients are obtained directly from $G_k$ with no extra overhead, though computing the full $G_k$ each iteration is more costly than subspace-only gradients.  

Then, an {\it exploitation} step in the subspace by gradient descent with learning rate $\eta_k$ gives
\begin{equation}\label{eq:memory_efficient_subspace_min}
    \begin{aligned}
    W_{k+1} 
    &= \tilde{W}_k + \mathcal{M}_k(R_k - \eta_k P_k^\top G_k) \\
    &= \tilde{W}_k + P_k(R_k - \eta_k P_k^\top G_k) \\
    &= W_k - \eta_k P_k P_k^\top G_k,
    \end{aligned}
\end{equation}
which matches the classical projected subspace descent step \citep{kozak2019stochastic}.  

In GaLore, $P_k$ is chosen as the rank-$r$ left SVD of the full gradient at fixed intervals.  
PESO recovers GaLore when the subspace gradient $\nabla_R \ell(W_k) = P_k^\top G_k$ in~\eqref{eq:memory_efficient_subspace_min} is replaced with its Adam update.  

This shows how GaLore saves memory: gradients of $\xi_k$ are derived directly from $G_k$, and updates are written back into $W$ via~\eqref{eq:memory_efficient_subspace_min}, reusing the stored pretrained weights $W_0$.  
Thus explicit storage of $\xi_k$ is unnecessary.  
Further savings arise because GaLore omits optimizer states for $\mathcal{M}_k$ (i.e., for $P_k$), instead updating $P_k$ by direct reassignment in a restart manner.
However, smoother transitions of subspace parameters often yield greater stability, as observed in our experiments; we discuss smoothing techniques for restart strategies in the next section.

\section{Implementation of PESO-LoRA-R}\label{sec:implementation}
This section details implementation strategies for \texttt{PESO-LoRA-R} that are critical for practical stability, specifically regarding subspace smoothing and the alignment of optimizer states. 
While these techniques collectively enhance robustness, our empirical observations indicate that \textit{velocity alignment} is the most significant factor in stabilizing training dynamics during restarts.
Consequently, for a streamlined implementation that faithfully recovers our results, we strongly recommend prioritizing the velocity alignment mechanism described in Section~\ref{subsec:alignment}.

\subsection{Smoothing the Subspaces}\label{subsec:smoothing_subspace}
A potential issue with restart methods (see~\eqref{eq:restart}) is that they directly reassign the subspace parameterization from new information, which can introduce sharp changes and instability, especially in LLM training where stochastic noise is significant.

To mitigate this, we adopt an Exponential Moving Average (EMA) of old and new subspaces, similar in spirit to how Adam \citep{kingma2014adam} stabilizes noisy updates.  
However, this is nontrivial in {\tt PESO-LoRA-R} (Algorithm~\ref{alg:lora_restart}), since the pre-restart adapters $(A_k,B_k)$—evolved through Adam dynamics—may differ significantly in scale and coordinates from the restarted pair $(U_k\sqrt{\Lambda_k},\sqrt{\Lambda_k}V_k)$ obtained from rank-$r$ SVD.  
A naive EMA would mismatch these terms and discard valuable exploration information.

We resolve this by performing {\it basis and scaling alignment}.  
For clarity, we omit the subscript $k$.  
Given current adapters $(A,B)$, we first compute thin QR factorizations:
\[
A=Q_A R_A,\quad B=R_B Q_B^\top, \quad Q_A^\top Q_A=I_r,\; Q_B^\top Q_B=I_r,
\]
to extract bases $(Q_A,Q_B)$ and decouple scaling.  
Next, let the rank-$r$ SVD of the full gradient be $-G \approx U\,\Sigma\,V^\top$.  
We align $(U,V)$ to $(Q_A,Q_B)$ by applying SVD to $Q_A^\top U$ and $Q_B^\top V$:
\[
Q_A^\top U = P_U \Sigma_U Q_U^\top,\;\; R_U:=P_U Q_U^\top, \qquad
Q_B^\top V = P_V \Sigma_V Q_V^\top,\;\; R_V:=P_V Q_V^\top,
\]
yielding aligned bases
\[
\widehat U := U R_U^\top,\qquad \widehat V := V R_V^\top.
\]
Here $R_U$ solves
\[
\min_R \|UR-Q_A\|_F,\quad RR^\top=I,
\]
so $\widehat U$ best aligns $U$ with $Q_A$ in Frobenius norm; the same holds for $\widehat V$.  
This produces the best alignment by classical low-rank SVD guarantees.

We then smooth the bases via EMA:
\[
U_{\mathrm{ema}} := \tau_1 Q_A + (1-\tau_1)\widehat U,\qquad
V_{\mathrm{ema}} := \tau_1 Q_B + (1-\tau_1)\widehat V,
\]
for smoothing parameter $\tau_1$.  
To smooth the scaling, we project the old adapter into the new bases and combine with the gradient:
\[
S_{\mathrm{new}} := \tau_2\,[U_{\mathrm{ema}}^\top (AB)\,V_{\mathrm{ema}}] - (1-\tau_2)\,[U_{\mathrm{ema}}^\top G\,V_{\mathrm{ema}}],
\]
with parameter $\tau_2$.  
Since this merges the scaling of $A$ and $B$, we refactorize $S_{\mathrm{new}}$ using polar decomposition \citep{glentis2025scalable}:
\[
S_{\mathrm{new}} = R_L\,\Sigma\,R_R^\top,
\]
and define the new adapters via balanced splitting:
\[
A_{\mathrm{new}} := U_{\mathrm{ema}}\,R_L\,\Sigma^{1/2},\qquad
B_{\mathrm{new}} := \Sigma^{1/2}\,R_R^\top\,V_{\mathrm{ema}}^\top.
\]
Finally, although empirically less effective, we also rotate the momentum vectors according to the new basis.  
Specifically, we compute
\[
T_A := Q_A^\top U_{\mathrm{ema}},\qquad T_B := Q_B^\top V_{\mathrm{ema}},
\]
and update the pre-restart momentum $(m_A,m_B)$ as
\[
m_A \leftarrow m_A T_A,\qquad m_B \leftarrow T_B^\top m_B.
\]
Training then proceeds with the new adapters $(A_{\mathrm{new}},B_{\mathrm{new}})$ in {\tt PESO-LoRA-R}.

\subsection{Momentum and Velocity Alignment}\label{subsec:alignment}
Even with basis and scaling alignment from the previous subsection, another stability issue arises: the new adapters $(A_{\mathrm{new}},B_{\mathrm{new}})$ can induce gradients of very different magnitudes compared to the old $(A,B)$.  
Since restarts are based on the SVD of the full gradient, the new adapters align with top gradient directions, so the gradients with respect to $(A_{\mathrm{new}},B_{\mathrm{new}})$ are typically larger.  
This mismatch can leave the velocity ``too cold'': historical states $(v_A,v_B)$ may underestimate the new gradient magnitudes, leading to an excessively large normalized step and unstable behavior, often observed as jumps in the loss curve.

To address this, we propose a combined \emph{momentum/velocity scaling} technique with a {\it $\beta_2$ warm-up}.  
Let $(m_A,m_B)$ and $(v_A,v_B)$ denote the momentum and velocity before restart, and $(g_A,g_B)$ the gradients after restart computed with respect to $(A_{\mathrm{new}},B_{\mathrm{new}})$.  
We define scaling factors
\[
s^{(v)}_A \;=\; \frac{\|g_A\|^2}{{\|v_A\|}},\quad
s^{(m)}_A \;=\; \frac{\|g_A\|}{\|m_A\|},\quad
s^{(v)}_B \;=\; \frac{\|g_B\|^2}{{\|v_B\|}},\quad
s^{(m)}_B \;=\; \frac{\|g_B\|}{\|m_B\|},
\]
which correct scale mismatches between the old optimization states and the new gradients.
Here, $\|\cdot\|$ denotes the RMS norm.
Momentum and velocity are then rescaled as
\[
v_A \;\gets\; s_A^{(v)}\,v_A,\quad
m_A \;\gets\; s_A^{(m)}\,m_A,\quad
v_B \;\gets\; s_B^{(v)}\,v_B,\quad
m_B \;\gets\; s_B^{(m)}\,m_B.
\]

This resolves scale mismatches, but an additional adjustment is needed: $\beta_2=0.999$ (velocity EMA) adapts much more slowly than $\beta_1=0.9$ (momentum EMA).  
At initialization, bias correction balances these, but after a restart we require extra correction.  
We therefore decrease $\beta_2$ immediately after a restart and gradually warm it back to $0.999$ over a window $T$.  
If a restart occurs at iteration $t_r$, then for $t_r \le t \le t_r+T$ we set
\[
\beta_2(t) \;=\; \beta_{2,\min} + \big(\beta_{2,\text{final}} - \beta_{2,\min}\big)\,\tfrac{1}{2}\!\left(1 - \cos \tfrac{\pi (t-t_r)}{T}\right), 
\quad \beta_{2,\text{final}} = 0.999,
\]
and for $t > t_r+T$ we use $\beta_2 = 0.999$ as usual.
In our experiments, we set $\beta_{2,\min} = 0.95$ and $T=\lfloor K/3\rfloor$.

{
\section{Detailed Version of Algorithms}\label{app:algo_details}
In this section, we provide detailed versions of {\tt PESO-LoRA-R} (Algorithm~\ref{alg:lora_restart}) and {\tt PESO-LoRA-T} (Algorithm~\ref{alg:SVD_exploit}) which are implemented in the experiments of Section~\ref{sec:experiments}.
}

\begin{algorithm}[ht]
\caption{{{\tt PESO-LoRA-R}: {\bf PESO} with {\bf LoRA} and Subspace Explo{\bf R}ation} (Practical Version)}
\label{alg:peso-lora-r-practical}
\begin{algorithmic}[1]

    \Require Pre-trained parameters $W_0\in\mathbb{R}^{m\times n}$, frequency $K$, scale parameter $\gamma$, learning rate $\eta$, AdamW hyperparameters $(\beta_1,\beta_2,\epsilon,\lambda)$.
    \State Set $k \gets 1$, $\tilde W_0\gets W_0$, $A_0\gets 0$, $B_0\gets 0$.
    \State Initialize AdamW states $m_A\gets 0$, $v_A\gets 0$, $m_B\gets 0$, $v_B\gets 0$.
    \While{stopping criteria not satisfied}
        \If{$k-1\mod{K}=0$}
            \State $\tilde W_k\gets \tilde W_{k-1}+A_{k-1}B_{k-1}$.
            \State Compute stochastic full gradient $G_k$.
            \State $(U_k,\Lambda_k,V_k)\gets {\tt SVD}(-G_k)$. \Comment{Top-$r$ SVD of $G_k$}
            \State Set $(A_{k-1},B_{k-1})\gets (A_{\text{new}},B_{\text{new}})$ according to Appendix~\ref{subsec:smoothing_subspace}.
            \State Update $m_A,m_B,v_A,v_B,\beta_2$ according to Appendix~\ref{subsec:alignment}.
        \EndIf
        \State Compute gradients $g_{A,k} \gets \nabla_{A}\ell(\tilde W_k + A_{k-1}B_{k-1})$, $g_{B,k} \gets \nabla_{B}\ell(\tilde W_k + A_{k-1}B_{k-1})$.
        \State $m_A \gets \beta_1 m_A + (1-\beta_1) g_{A,k}$,\quad $v_A \gets \beta_2 v_A + (1-\beta_2) g_{A,k}\odot g_{A,k}$.
        \State $m_B \gets \beta_1 m_B + (1-\beta_1) g_{B,k}$,\quad $v_B \gets \beta_2 v_B + (1-\beta_2) g_{B,k}\odot g_{B,k}$.
        \State $\hat m_A \gets m_A / (1-\beta_1^k)$,\quad $\hat v_A \gets v_A / (1-\beta_2^k)$.
        \State $\hat m_B \gets m_B / (1-\beta_1^k)$,\quad $\hat v_B \gets v_B / (1-\beta_2^k)$.
        \State $A_k \gets A_{k-1} - \eta\lambda A_{k-1} - \eta\, \hat m_A / (\sqrt{\hat v_A} + \epsilon)$.
        \State $B_k \gets B_{k-1} - \eta\lambda B_{k-1} - \eta\, \hat m_B / (\sqrt{\hat v_B} + \epsilon)$.
        \State $k \gets k+1$.
    \EndWhile 
    \State \Return $\tilde W_k+A_kB_k$.

\end{algorithmic}
\end{algorithm}

\begin{algorithm}[ht]
\caption{{{\tt PESO-LoRA-T}: {\bf PESO} with {\bf LoRA} and Subspace Exploi{\bf T}ation} (Practical Version)}
\begin{algorithmic}[1]

    \Require Pretrained weights $W_0\in\mathbb{R}^{m\times n}$, initial subspace matrices $U_0\in\mathbb{R}^{m\times r}$, $V_0\in\mathbb{R}^{r\times n}$, initial coordinate $\xi_0\in\mathbb{R}^r$, frequency $K$, learning rate $\eta$, AdamW hyperparameters $(\beta_1,\beta_2,\epsilon,\lambda)$.
    \State Set $k \gets 1$.
    \State Initialize AdamW states $m_U,m_V,m_\xi \gets 0$ and $v_U,v_V,v_\xi \gets 0$.

    \While{stopping criterion not met}
        \State Keep $(U_{k},V_{k}) \gets (U_{k-1},V_{k-1})$.

        \If{$k-1 \bmod K = 0$}
            \State Compute gradient $g_{U,k} \gets \nabla_U \ell(W_0 + U_{k-1}\operatorname{diag}(\xi_{k-1})V_{k-1})$.
            \State Compute gradient $g_{V,k} \gets \nabla_V \ell(W_0 + U_{k-1}\operatorname{diag}(\xi_{k-1})V_{k-1})$.
            \State $m_U \gets \beta_1 m_U + (1-\beta_1) g_{U,k}$,\quad $v_U \gets \beta_2 v_U + (1-\beta_2) g_{U,k}\odot g_{U,k}$.
            \State $m_V \gets \beta_1 m_V + (1-\beta_1) g_{V,k}$,\quad $v_V \gets \beta_2 v_V + (1-\beta_2) g_{V,k}\odot g_{V,k}$.
            \State $\hat m_U \gets m_U/(1-\beta_1^k)$,\quad $\hat v_U \gets v_U/(1-\beta_2^k)$.
            \State $\hat m_V \gets m_V/(1-\beta_1^k)$,\quad $\hat v_V \gets v_V/(1-\beta_2^k)$.
            \State $U_k \gets U_{k-1} - \eta\lambda U_{k-1} - \eta\,\hat m_U / (\sqrt{\hat v_U}+\epsilon)$.
            \State $V_k \gets V_{k-1} - \eta\lambda V_{k-1} - \eta\,\hat m_V / (\sqrt{\hat v_V}+\epsilon)$.
        \EndIf

        \State Compute gradient $g_{\xi,k} \gets \nabla_{\xi}\ell(W_0 + U_{k}\operatorname{diag}(\xi_{k-1})V_{k})$.
        \State $m_\xi \gets \beta_1 m_\xi + (1-\beta_1) g_{\xi,k}$,\quad $v_\xi \gets \beta_2 v_\xi + (1-\beta_2) g_{\xi,k}\odot g_{\xi,k}$.
        \State $\hat m_\xi \gets m_\xi/(1-\beta_1^k)$,\quad $\hat v_\xi \gets v_\xi/(1-\beta_2^k)$.
        \State $\xi_k \gets \xi_{k-1} - \eta\lambda \xi_{k-1} - \eta\,\hat m_\xi / (\sqrt{\hat v_\xi}+\epsilon)$.

        \State $k \gets k+1$.
    \EndWhile

    \State \Return $W_0 + U_k \operatorname{diag}(\xi_k) V_k$.
\end{algorithmic}
\end{algorithm}

\section{Experimental Details}\label{app:experiments}
All experiments are conducted on NVIDIA RTX A6000 GPUs. For \texttt{PESO-LoRA-R}, to further reduce computational cost, we restrict the exploration frequency to at most two times per epoch.

\subsection{Natural Language Understanding: Hyperparameter Settings}\label{app:nlu_experiments}

In Section~\ref{subsec:nlu_tasks}, we present the results of our methods and various LoRA-based algorithms on natural language understanding tasks, following the prompt tuning configuration of~\citep{wang2024loraga}. The general hyperparameter settings are kept consistent across all algorithms which are shown in Table~\ref{tab:hyper_for_nlu}. To ensure a fair comparison, we follow~\citep{zhang2025one} and tune the learning rates via grid search over \(\{1\times 10^{-4}, 2\times 10^{-4}, 5\times 10^{-4}, 1\times 10^{-3}\}\). Additionally, following the choices of~\citet{zhang2025one}, the scale parameters for LoRA-One are set to be \{128, 16, 128, 128, 64\} for \{MNLI, SST-2, CoLA, QNLI, MRPC\}.

For \texttt{PESO-LoRA-R}, we set the smoothing parameter \(\tau_1=\tau_2=0.9\), with frequency \(K\) chosen as \(\{2000, 500, 100, 500, 40\}\) for \{MNLI, SST-2, CoLA, QNLI,
and MRPC\} based on empirical observations. When \((k-1) \bmod K = 0\) and \(k \neq 0\), we set the scale parameter \(\gamma = 1\); when \(k = 0\), the scale parameter is set the same as in LoRA-One.
 To further reduce computational cost, we restrict the restart frequency to times per epoch. For \texttt{PESO-LoRA-T}, we set frequency \(K=1\) for all datasets.

\begin{table}[ht]
\centering
\caption{Common hyperparameters for LoRA fine-tuning on T5-base model.}
\label{tab:hyper_for_nlu}
\resizebox{\textwidth}{!}{%
\begin{tabular}{ccccccc}
\toprule
Epoch & Optimizer & $(\beta_1, \beta_2)$ & $\epsilon$ & Batch Size & Weight Decay & LR Scheduler \\
\midrule
1 & AdamW & (0.9, 0.999) & $1 \times 10^{-8}$ & 32 & 0 & cosine \\
\midrule
Warm-up Ratio & LoRA Alpha & \#Runs & Sequence Length & Adapt Precision & Backbone Precision & Gradient Batch Size  \\
\midrule
0.03 & 16 & 3 & 128 & FP32 & FP32 & 8 \\
\bottomrule
\end{tabular}%
}
\end{table}

\subsection{Natural Language Generation: Hyperparameter Settings}\label{app:nlg_experiments}

\begin{table}[tb]
\centering
\caption{Performance of fine-tuned Llama-2-7B on natural language generation tasks with rank set to 8. Results are reported as accuracy (\%) over 3 runs. \textbf{Bold} and \underline{underline} indicate the highest and second-highest accuracies \emph{excluding} {\tt PESO-LoRA-T}.
}
\vskip -0.1in
\label{tab:result_nlg}
\centering
\begin{tabular}{c|ccccG}
\toprule
 & LoRA & LoRA-GA & LoRA-One & {\tt PESO-LoRA-R} & {\tt PESO-LoRA-T} \\
\midrule
GSM8K     
& 59.26$_{\pm0.99}$ 
& 56.44$_{\pm1.15}$ 
& \underline{60.44}$_{\pm0.17}$ 
& \textbf{60.55}$_{\pm0.34}$ 
& 60.82$_{\pm0.77}$ \\
MMLU        
& 45.73$_{\pm0.30}$ 
& 45.15$_{\pm0.57}$ 
& \textbf{47.24}$_{\pm0.20}$ 
& \underline{46.16}$_{\pm0.58}$ 
& 46.44$_{\pm0.37}$ \\
HumanEval   
& 25.85$_{\pm1.75}$ 
& 26.95$_{\pm1.30}$ 
& \underline{28.66}$_{\pm0.39}$ 
& \textbf{31.70}$_{\pm1.30}$ 
& 30.85$_{\pm1.18}$ \\
\bottomrule
\end{tabular}
\end{table}

\subsubsection{Comparison with LoRA-based methods}
For natural language generation tasks in Section~\ref{sec:nlg_tasks}, we follow the configuration of prompt tuning and strategy of hyperparameter tuning in~\citep{zhang2025one} to ensure fair comparison. We search the best learning rate over 
\(\{5 \times 10^{-4},\, 2 \times 10^{-4},\, 1 \times 10^{-4},\, 5 \times 10^{-5},\, 2 \times 10^{-5},\, 1 \times 10^{-5}\}\), and the general hyperparameter setting is summarized in Table~\ref{tab:hyper_for_nlg}. 
Additionally, following the choice of~\citet{zhang2025one}, the scale parameters are tuned within $\{16, 32, 64,128\}$ for LoRA-GA, LoRA-One and our methods to achieve the best performances.

\begin{table}[ht]
\centering
\caption{Common hyperparameters for LoRA fine-tuning on Llama-2-7B and LLama-3.1-8B model.}
\label{tab:hyper_for_nlg}
\resizebox{\textwidth}{!}{%
\begin{tabular}{ccccccc}
\toprule
Epoch & Optimizer & $(\beta_1, \beta_2)$ & $\epsilon$ & Batch Size & Weight Decay & LR Scheduler \\
\midrule
1 & AdamW & (0.9, 0.999) & $1 \times 10^{-8}$ & 32 & 0 & cosine \\
\midrule
Warm-up Ratio & LoRA Alpha & \#Runs & Sequence Length & Adapter Precision & Backbone Precision &  Gradient Batch Size  \\
\midrule
0.03 & 16 & 3 & 1024 & FP32 & BF16 & 8  \\
\bottomrule
\end{tabular}%
}
\end{table}

For \texttt{PESO-LoRA-R}, we set the smoothing parameter \(\tau_1=\tau_2=0.9\), with frequency \(K=500\) for all the experiments. When \((k-1) \bmod K = 0\) and \(k \neq 0\), we set the scale parameter \(\gamma = 1\); when \(k = 0\), the scale parameter is set the same as in LoRA-One.
For \texttt{PESO-LoRA-T}, we set frequency \(K=1\) for all datasets.

{\subsubsection{Comparison with pretraining-oriented subspace methods} \label{app:compare_with_pretraining}}

{For comparison with pretraining-oriented subspace-based approaches, we follow the experimental setup described in \citet{zhu2024apollo}. 
Specifically, we fine-tune Llama-3-8B-Instruct on the Alpaca-en-demo dataset for 3 epochs and evaluate the resulting models on the MMLU subtasks. 
For GaLore~\citep{zhao2024galore}, Fira~\citep{chen2024fira}, APOLLO~\citep{zhu2024apollo}, and our \texttt{PESO-LoRA-R}, we report the best accuracy 
achieved by sweeping the learning rate over
\(
\{5\mathrm{e}{-6},\, 7.5\mathrm{e}{-6},\, 1\mathrm{e}{-5},\, 2.5\mathrm{e}{-5},\, 5\mathrm{e}{-5},\, 7.5\mathrm{e}{-5},\, 1\mathrm{e}{-4},\, 1.5\mathrm{e}{-4},\, 2\mathrm{e}{-4}\}.
\)
All methods use a fixed rank of $r=8$. For \texttt{PESO-LoRA-R}, the scale parameter is set to $\gamma=16$.

The results are shown in Table~\ref{tab:subspace_comparison}, demonstrating that our PEFT-specialized method \texttt{PESO-LoRA-R} consistently outperforms these pretraining-oriented baselines in the fine-tuning regime.
}

% \begin{table*}[!htbp]
% \centering
% \caption{Comparison with pretraining-oriented subspace methods.
% We fine-tune Llama-3-8B-Instruct on Alpaca-en-demo with rank $r=8$
% and evaluate on MMLU subtasks. Results are reported as accuracy (\%).
% For GaLore, Fira, and APOLLO, we use the results reported in \citet{zhu2024apollo}.
% Following their convention, we report our method’s average accuracy over 3 runs.
% \textbf{Bold} and \underline{underline} indicate the highest and second-highest accuracies.}
% \label{tab:subspace_comparison}

% \begin{tabular}{c|ccccc}
% \toprule
% Method & STEM & Social Sciences & Humanities & Others & Average \\
% \midrule
% GaLore        & 54.50 & 75.11 & 58.59 & 72.03 & 64.43 \\
% Fira          & 53.53 & 75.46 & 58.59 & 72.09 & 64.32 \\
% APOLLO /w SVD & \underline{54.73} & 75.46 & \underline{58.72} & \underline{72.68} & \underline{64.76} \\
% APOLLO        & 54.37 & \underline{75.86} & 58.18 & 71.69 & 64.35 \\
% \midrule
% \texttt{PESO-LoRA-R}
% & \textbf{57.15} & \textbf{76.86} & \textbf{60.79} & \textbf{73.57} & \textbf{66.47} \\
% \bottomrule
% \end{tabular}
% \end{table*}

{\subsection{Natural Language Generation: Ablation Study and Computational Cost}\label{app:nlg_experiments_ablation}}

{We conduct an ablation study on \texttt{PESO-LoRA-R} by fine-tuning Llama-3.1-8B and evaluating the model on GSM8K. 
We vary the rank $r$ and the restart frequency $K$, and additionally compare using one versus two restarts per epoch. 
We set the scale parameter to $\gamma = 128$, and all other hyperparameters, unless otherwise noted, follow Appendix~\ref{app:nlg_experiments}. 
The results are presented in Table~\ref{tab:ablation_rank} and~\ref{tab:ablation_restart}.}

{Overall, the results indicate that \texttt{PESO-LoRA-R} does not require a large rank to achieve strong performance: accuracy 
improves substantially even at small ranks, with diminishing gains beyond $r=8$. Furthermore, the method shows robust behavior 
across a wide range of restart frequencies, as performance remains stable when varying $K$ or the number of restarts performed 
per epoch.}

{We also report memory and runtime for Llama-3.1-8B fine-tuning on GSM8K with rank $r=8$, $K=500$, and restricting one restart per epoch. As shown in Table~\ref{tab:ablation_cost}, \texttt{PESO-LoRA-R} matches the memory footprint of vanilla LoRA and incurs only a negligible increase in runtime. Specifically, the SVD (incurred when doing restart) step accounts for only 1.9\% of the total computation time in our experiments, indicating that its overhead is negligible.}

\begin{table}[t]
\centering
\caption{Ablation on rank $r$ for \texttt{PESO-LoRA-R}, evaluated on GSM8K with restart frequency $K=500$.}
\label{tab:ablation_rank}
\begin{tabular}{c|ccccccc}
\toprule
Rank $r$ & 1 & 2 & 4 & 8 & 16 & 32 & 64 \\
\midrule
Acc (\%) & 75.28 & 75.97 & 76.88 & 77.56 & 77.71 & 78.01 & 78.39 \\
\bottomrule
\end{tabular}
\end{table}

\begin{table}[t]
\centering
\caption{Ablation on restart frequency $K$ and restart count (once or twice per epoch), evaluated on GSM8K with rank $r=8$. 
Infinity denotes applying exploration only at initialization.}
\label{tab:ablation_restart}
\begin{tabular}{c|ccccc}
\toprule
$K$ & 250 & 500 & 750 & 1000 & Infinity \\
\midrule
Restart Once  & 77.18 & 77.56 & 77.18 & 77.26 & 77.26 \\
Restart Twice & 76.88 & 77.30 & 77.56 & 76.88 & -- \\
\bottomrule
\end{tabular}
\end{table}

\begin{table}[t]
\centering
\caption{Memory (excluding transient peaks) and time for Llama-3.1-8B fine-tuning on GSM8K with $r=8$, $K=500$, and one restart per epoch.}
\label{tab:ablation_cost}
\begin{tabular}{c|cc}
\toprule
Method & LoRA & \texttt{PESO-LoRA-R} \\
\midrule
Time   & 5h03m & 5h09m \\
Memory & 25.7G & 25.7G \\
\bottomrule
\end{tabular}
\end{table}

\subsection{Multi-Epoch Low-Rank Analysis: Hyperparameter Settings}
In Section~\ref{sec:long_epoch}, we fine-tune the T5-base model on SST-2 dataset for 4 epochs. We vary the rank of LoRA in \(\{2,4,8\}\), keep the rank of \texttt{PESO-LoRA-R} as 2, and add full-parameter fine-tuning for comparison. We keep all the other hyperparameter settings the same as in~\ref{app:nlu_experiments}.

\section{Asymptotic Convergence}\label{sec:asym}
In this section, we analyze the asymptotic convergence of the general PESO framework (Algorithm~\ref{alg:generic_subspace}). 
To maintain generality, we do not restrict the specific choices of subspaces $\{\mathcal{S}_k\}$ or the subroutines {\tt Opt} and {\tt UpdateM}. 
Consequently, the derived convergence bounds may contain a bias term driven by the alignment quality between the subspace and the true gradient.
We formalize this quality by assuming the subspace $\mathcal{S}_k$ sufficiently captures the gradient $G_k$ at restart indices.

\begin{assumption}\label{assump:accurate_subspace}
    There exists a sequence $\{\delta_k\}\geq 0$ such that $dist(G_k,\mathcal{S}_k)\leq \delta_k$ for $k$ where full gradient restart is implemented.
    Furthermore, $\lim_{k\to\infty}
    \delta_k<\infty$.
\end{assumption}

\begin{rem}
    In the general case, if $\liminf_{k\to\infty} \delta_k > 0$, the convergence established in this section may contain a bias term depending on the approximation error. 
    However, this bias can be eliminated in specific instantiations:
    (i) if $\mathcal{S}_k$ is chosen such that $G_k \in \mathcal{S}_k^{\text{FG}}$, then $\delta_k=0$ and the optimality measure converges to zero.
    (ii) in a deterministic setting, if $\mathcal{S}_k$ is chosen as the top-$r$ SVD subspace of $G_k$ (as in {\tt PESO-LoRA-R}), the bias term vanishes. In this case, we establish exact asymptotic convergence, i.e., $\liminf_{k\to\infty} \|G_k\| = 0$. We provide the full details in Appendix~\ref{app:additional_theory}.
    (iii) Even in the stochastic case, exact convergence can be recovered. By specifying the subspace construction as top-$r$ SVD (to control projection bias) and adopting a carefully tuned hyperparameter schedule, {\tt PESO-LoRA-R} achieves exact convergence. The proof of this result (Theorem~\ref{thm:peso_rate}) is provided in Appendix~\ref{sec:proof_details}.
\end{rem}

\subsection{Deterministic Case}
For completeness, we first consider the deterministic setting, where the gradients $G_k$ used by PESO (Algorithm~\ref{alg:generic_subspace}) are exact and free of stochastic noise.  
We begin by stating a mild assumption on the behavior of {\tt Opt} and {\tt UpdateM}, which ensures that each iteration produces a descent step.

\begin{assumption}\label{assump:opt_descent}
    {\tt Opt} and {\tt UpdateM} generate the updates satisfying $\ell(W_k)\leq \ell(\tilde W_{k}+\mathcal{M}_k(\xi_{k-1}))$ and $\ell(\tilde W_{k}+\mathcal{M}_{k}(\xi_{k-1})) \leq \ell(W_{k-1})$ for all $k=1,2,\cdots$.
\end{assumption}

Assumption~\ref{assump:opt_descent} requires that both {\tt Opt} and {\tt UpdateM} are descent-preserving, ensuring that the objective value is non-increasing across iterations.
This condition is standard and holds, for example, when both procedures are implemented via gradient descent with step size $\alpha \le 1/L$, where $L$ is the Lipschitz constant of the gradient as specified in Assumption~\ref{assump:loss_regular}.  
In particular, {\tt UpdateM} can be implemented as a warm-started gradient step on the subspace parameters with respect to the original loss function, initialized from the previous iterate.  
Such descent guarantees follow from classical smooth optimization results; see, e.g., \citet{nesterov2018lectures}.

\begin{prop}
    Let {\tt Opt} and {\tt UpdateM} be gradient descent schemes on $\ell$ with constant learning rate $\alpha \le 1/L$.  
    Then Assumption~\ref{assump:opt_descent} holds.
\end{prop}

%If $\MMM_k$ is linear with respect to $\xi$ (e.g. the PCA algorithm above), then a gradient descent algorithm with step size $\frac{1}{L}$ for {\tt Opt} can achieve this as well.

We now present the deterministic convergence result.

\begin{thm}\label{thm:conv_descent_lem}
    Suppose Assumptions~\ref{assump:loss_regular}, \ref{assump:accurate_subspace}, and \ref{assump:opt_descent} hold.  
    With full gradient restart enabled and learning rate $\eta_k=\tfrac{1}{L}$, the iterates $\{W_k\}$ generated by Algorithm~\ref{alg:generic_subspace} satisfy 
    $\liminf_{k\to\infty} \|G_k\|\leq \lim_{k\to\infty}\delta_k$.
\end{thm}

\begin{proof}
We assume the frequency of the full gradient restart is $K$.
By the descent lemma, whenever $k-1 \mod K = 0$ (i.e., when a full gradient restart occurs), define
\begin{equation}
    \hat W_{k}:=W_{k-1}-\tfrac{1}{L}P_{\SSS_k}(G_k).
\end{equation}
Performing a full gradient restart as {\tt UpdateM}, as discussed in~\eqref{eq:grad_restart_generic}, is thus equivalent to moving from $W_{k-1}$ to $\hat W_k$.  
It follows that
\begin{equation}
\begin{split}
     \ell(\hat W_{k})=\ell(W_{k-1}-\tfrac{1}{L}P_{\SSS_k}(G_k)) 
     & \leq \ell(W_{k-1})+\left\langle G_k,-\tfrac{1}{L}P_{\SSS_k}(G_k) \right\rangle+\tfrac{L}{2}\|\tfrac{1}{L}P_{\SSS_k}(G_k)\|^2 \\
     & = \ell(W_{k-1})-\tfrac{1}{L}\|P_{\SSS_k}(G_k)\|^2+\tfrac{L}{2}\|\tfrac{1}{L}P_{\SSS_k}(G_k)\|^2\\
     & = \ell(W_{k-1})-\tfrac{1}{2L}\|P_{\SSS_k}(G_k)\|^2,
\end{split}
\end{equation}
where the second equality holds because projection onto a subspace is orthogonal.  

Therefore, by Assumption~\ref{assump:opt_descent}, for iterates $i=k,\ldots,k+K-1$ (note that $\ell(W_k)\leq \ell(\hat W_k)$ since $W_k$ is obtained by applying {\tt Opt} to $\hat W_k$),
\begin{equation}\label{eq:descent_K2}
    \tfrac{1}{2L}\|P_{\SSS_k}(G_k)\|^2\leq \ell(W_{k-1})-\ell(\hat W_{k})\leq \ell(W_{k-1})-\ell(W_i).
\end{equation}
Importantly,~\eqref{eq:descent_K2} remains valid regardless of how frequently other types of {\tt UpdateM} are applied between full gradient restarts, since all updates preserve the descent property by Assumption~\ref{assump:opt_descent}.  
In particular, whenever $\mathcal{M}_k$ is updated without a full gradient restart, we have
\begin{equation}
\begin{split}
        \ell(W_{k-1})-\ell(W_{k}) 
        & =\ell(\tilde W_{k-1}+\mathcal{M}_{k-1}(\xi_{k-1}))-\ell(W_k) \\
        & \geq \ell(\tilde W_{k-1}+\mathcal{M}_{k-1}(\xi_{k-1}))-\ell(\tilde W_{k}+\mathcal{M}_{k}(\xi_{k-1})) \\
        & \geq 0.
\end{split}
\end{equation}
This ensures the chain of inequalities in~\eqref{eq:descent_K2} continues to hold when updates are performed by {\tt OptM}.  

Hence, for any integer $k \in \mathbb{N}$,
\begin{equation}
    \tfrac{1}{2L}\|P_{\SSS_{kK+1}}(G_{kK+1})\|^2\leq \ell(W_{kK})-\ell(W_{(k+1)K}).
\end{equation}
Here $kK$ and $(k+1)K$ denote integer products.  

Since $\{\ell(W_k)\}$ is bounded below (Assumption~\ref{assump:loss_regular}) and monotonically decreasing (Assumption~\ref{assump:opt_descent} together with the descent lemma at restart points), it converges by the monotone convergence theorem and is Cauchy.  
Thus $\ell(W_k)-\ell(W_{k+1})\to 0$, and 
\begin{equation}\label{eq:projection_converge}
    \tfrac{1}{2L}\|P_{\SSS_{kK+1}}(G_{kK+1})\|^2\to 0.
\end{equation}
Finally, note that
\begin{equation}
    \|G_{kK+1}\|\leq \mathrm{dist}(G_{kK+1},\mathcal{S}_{kK+1})+\|P_{\SSS_{kK+1}}(G_{kK+1})\|
\leq \delta_{kK+1}+\|P_{\SSS_{kK+1}}(G_{kK+1})\|\to \delta,
\end{equation}
where $\delta:=\lim_{k\to\infty}\delta_k$.  
Therefore, $\liminf_{k\to\infty}\|G_k\|\leq \delta$.
\end{proof}

\subsection{Exact Convergence of Deterministic {\tt PESO-LoRA-R}}\label{app:additional_theory}
% There has been a growing body of work establishing exact convergence guarantees for memory-efficient pre-training methods, e.g., \citep{chen2025memory,he2024subspace,robert2024ldadam,zhang2025breaking}.
% These results typically specify concrete subspace constructions (such as randomized projections), under which the bias term depending on $\delta_k$ can be eliminated, provided that the algorithm adopts a carefully tuned hyperparameter schedule, for example a $\beta_1$ or learning rate schedule that depends on the iteration horizon.
% While such schedules are theoretically valid, they often require problem- or time-dependent tuning, which limits their practicality in standard PEFT settings.

In this section, we present a simple case study in which {\tt PESO-LoRA-R} admits exact asymptotic convergence. 
We consider the deterministic setting where the subspace $\mathcal{S}_k^{\text{FG}}$ is chosen as the rank-$r$ subspace spanned by the top singular vectors of $G_k$, consistent with the construction used in {\tt PESO-LoRA-R}.  
Under this setting, we show that the algorithm converges asymptotically to a stationary point without approximation error induced by subspace restriction. 
This result is consistent with the result in \citet{he2024subspace}, which establishes exact convergence for GaLore under deterministic gradients.

We first present the following lemma, which bounds the error introduced by the projected gradients.
\begin{lem}\label{lem:twosided}
Let $G\in\mathbb{R}^{m\times n}$ have singular value decomposition
\[
G = U \Sigma V^\top,\qquad 
\Sigma = \mathrm{diag}(\sigma_1,\dots,\sigma_p),\quad
\sigma_1 \ge \cdots \ge \sigma_p \ge 0,
\]
where $p=\min\{m,n\}$. Let
\[
U_r := U[:,1:r],\quad
V_r := V[:,1:r],\quad
\Lambda_r := \mathrm{diag}(\sigma_1,\dots,\sigma_r),
\]
and define $G_r := U_r \Lambda_r V_r^\top$. Then the following bounds hold:
\begin{equation}\label{eq:twosided-bound}
\|G - G_r\|_F^2\;\le\;\Bigl(1 - \frac{r}{p}\Bigr)\,\|G\|_F^2
\quad\text{and}\quad
\|G_r\|_F^2\;\ge\;\frac{r}{p}\,\|G\|_F^2.
\end{equation}
\end{lem}

\begin{proof}
By the Eckart--Young--Mirsky theorem,
\[
\|G - G_r\|_F^2
= \sum_{i=r+1}^p \sigma_i^2,
\qquad
\|G\|_F^2 = \sum_{i=1}^p \sigma_i^2,
\qquad
\|G_r\|_F^2 = \sum_{i=1}^r \sigma_i^2.
\]
Since the singular values are sorted in nonincreasing order, we have
\[
\frac{1}{r}\sum_{i=1}^r \sigma_i^2
\;\ge\;
\frac{1}{p}\sum_{i=1}^p \sigma_i^2,
\]
which implies the second inequality in \eqref{eq:twosided-bound}:
\[
\|G_r\|_F^2 = \sum_{i=1}^r \sigma_i^2
\;\ge\;
\frac{r}{p}\sum_{i=1}^p \sigma_i^2
= \frac{r}{p}\,\|G\|_F^2.
\]
The first inequality follows immediately from the Pythagorean identity $\|G\|_F^2 = \|G_r\|_F^2 + \|G - G_r\|_F^2$:
\[
\|G - G_r\|_F^2
= \|G\|_F^2 - \|G_r\|_F^2
\;\le\;
\|G\|_F^2 - \frac{r}{p}\,\|G\|_F^2
=
\Bigl(1 - \frac{r}{p}\Bigr)\,\|G\|_F^2.
\]
\end{proof}

Next, we state the exact convergence result.

\begin{thm}\label{thm:conv_exact}
Suppose Assumption~\ref{assump:loss_regular} and \ref{assump:opt_descent} hold.
Assume that full-gradient restarts are enabled, the learning rate is chosen as $\eta_k = \tfrac{1}{L}$, and the subspace $\mathcal{S}_k^{\mathrm{FG}}$ is selected as the top-$r$ SVD subspace of $G_k$, as in {\tt PESO-LoRA-R}.
Then the iterates $\{W_k\}$ generated by Algorithm~\ref{alg:generic_subspace} satisfy $\liminf_{k\to\infty}\|G_k\| = 0$.
\end{thm}

\begin{proof}
The proof follows the same structure as that of Theorem~\ref{thm:conv_descent_lem}, with one key modification.
Since now $\mathcal{S}_{kK+1}$ is chosen as the top-$r$ SVD subspace of $G_{kK+1}$,
\begin{equation} 
\begin{split} 
\|G_{kK+1}\| & = \|G_{kK+1}-P_{\SSS_{kK+1}}(G_{kK+1})+P_{\SSS_{kK+1}}(G_{kK+1})\| \\ 
&\leq \sqrt{1 - \frac{r}{p}}\,\|G_{kK+1}\|+\|P_{\SSS_{kK+1}}(G_{kK+1})\|, 
\end{split} 
\end{equation} 
where $p:=\min\{m,n\}$ and the inequality follows from Lemma~\ref{lem:twosided} and the triangle inequality. 
It follows from~\eqref{eq:projection_converge} that 
\[ \left( 1-\sqrt{1 - \frac{r}{p}} \right) \|G_{kK+1}\|\leq \|P_{\SSS_{kK+1}}(G_{kK+1})\|\to 0, \] 
and therefore $\liminf_{k\to\infty} \|G_k\|=0$.
\end{proof}

\subsection{Stochastic Case}

We now turn to the stochastic setting, where only a stochastic estimator $\tilde G_k$ of the true gradient $G_k$ is available.
Our goal is to establish asymptotic convergence under standard assumptions on the learning rates and the stochasticity of the updates.
We begin by stating a condition on the learning rates used in the full-gradient restart step.

\begin{assumption}\label{assump:lr_stochastic}
     The learning rates for full gradient restart in~\eqref{eq:grad_restart_generic} satisfies $\sum \eta_k=\infty$ and $\sum \eta_k^2<\infty$.
\end{assumption}

Assumption~\ref{assump:lr_stochastic} is standard in stochastic optimization; see, e.g., \citet{bottou2018optimization}.  
We next impose a stochastic descent condition on the updates generated by {\tt Opt} and {\tt UpdateM}.

% We consider the stochastic setting where instead of access to the true gradient $G_k$, we only observe its stochastic estimate, denoted by $\tilde{G}_k$. 
% Following \citep{bottou2018optimization}, we make the following assumptions on the randomness in $\tilde{G}_k$ and the learning rates, which are common in the stochastic optimization literature.
\begin{assumption}\label{assump:descent_exp}
    {\tt Opt} and {\tt UpdateM} generate the updates satisfying $\mathbb{E}[\ell(W_{k})]\leq \mathbb{E}[\ell(\tilde W_k+\mathcal{M}_k(\xi_{k-1}))]+C_{k}$ and $\mathbb{E}[\ell(\tilde W_{k}+\mathcal{M}_{k}(\xi_{k-1}))] \leq \mathbb{E}[\ell(W_{k-1})]+C_k$ where $\sum_k|C_k|<\infty$.
\end{assumption}
Assumption~\ref{assump:descent_exp} can be viewed as a stochastic analogue of Assumption~\ref{assump:opt_descent}, allowing for controlled noise in the descent property.
This assumption is mild and is satisfied, for instance, when both {\tt Opt} and {\tt UpdateM} are implemented using stochastic gradient descent with diminishing step sizes; see below.

We begin by verifying the validity of Assumption~\ref{assump:descent_exp}.  
As an illustrative case, suppose {\tt Opt} and {\tt UpdateM} are implemented by SGD with diminishing step sizes $\{\alpha_k\}$ satisfying $\sum_k \alpha_k^2 < \infty$.  
Let $\hat{W}_k$ denote the weight after such an update.  
By the descent lemma (see also \citep[Lemma 4.4]{bottou2018optimization}), the expected decrease can be bounded as
\begin{equation}
    \mathbb{E}[\ell (\hat W_{k})] 
\;\le\; \mathbb{E}[\ell(W_k)]
- \alpha_k\!\left(1 - \tfrac{L\alpha_k}{2}\right)\!\mathbb{E}\|G_k\|^2
+ \tfrac{L}{2}\alpha_k^2 C.
\end{equation}
For $\alpha_k \le 1/L$, this simplifies to
\begin{equation}
    \mathbb{E}[\ell (\hat W_{k})] -\mathbb{E}[\ell(W_k)]
\;\le\; -\tfrac{\alpha_k}{2}\,\mathbb{E}\|G_k\|^2
+ \tfrac{L}{2}\alpha_k^2C.
\end{equation}
Taking positive and negative parts, we obtain
\begin{equation}
    \big[\mathbb{E}[\ell(\hat W_{k})] - \mathbb{E}[\ell(W_k)]\big]_+
\;\le\; \tfrac{L}{2}C\alpha_k^2,
\qquad
\big[\mathbb{E}[\ell(W_k)] - \mathbb{E}[\ell(\hat W_{k})]\big]_+
\;\le\; \mathbb{E}[\ell(W_k)] - \mathbb{E}[\ell(\hat W_{k})].
\end{equation}
Summing over all $k$,
\begin{equation}
\begin{split}
    \sum_{k=0}^\infty \big[\mathbb{E}[\ell(\hat W_{k})] - \mathbb{E}[\ell(W_k)]\big]_+
    &\;\le\; \tfrac{L}{2}C\sum_{k=0}^\infty \alpha_k^2 \;<\;\infty, \\
    \sum_{k=0}^\infty \big[\mathbb{E}[\ell(W_k)] - \mathbb{E}[\ell(\hat W_{k})]\big]_+
    &\;\le\; \sup_{k}\mathbb{E}[\ell(W_k)] - \ell(W^*) \;<\;\infty.
\end{split}
\end{equation}

Defining $C_k := \mathbb{E}[\ell (\hat W_{k})] - \mathbb{E}[\ell(W_k)]$, we conclude that $\sum_{k}|C_k| < \infty$, hence Assumption~\ref{assump:descent_exp} holds.

We are now ready to present the proof of our main result, Theorem~\ref{thm:conv_stochastic}.
Importantly, the result holds for {any choice} of {\tt UpdateM}, whether warm-start or restart, and at any frequency, as long as the assumptions are satisfied.
%, provided Assumption~\ref{assump:descent_exp} is satisfied and full gradient restart is enabled.  
In particular, Algorithm~\ref{alg:generic_subspace} may combine different {\tt UpdateM} strategies at varying frequencies, and the guarantee remains valid.

\begin{thm}\label{thm:conv_stochastic}
   Suppose Assumption~\ref{assump:loss_regular},~\ref{assump:moments},~\ref{assump:accurate_subspace},~\ref{assump:lr_stochastic},~\ref{assump:descent_exp} hold.
   In addition, assume
   \begin{equation}\label{assump:unbiased_conditional}
       \mathbb{E}[\tilde G_k \mid W_{k-1}] = G_k \qquad \text{for all } k,
   \end{equation}
   and that $\SSS_k$ is measurable with respect to $W_{k-1}$.
   With full gradient restart, the iterates $\{W_k\}$ generated by Algorithm~\ref{alg:generic_subspace} satisfy
   \[
   \liminf_{k\to\infty} \mathbb{E}[\|G_k\|] \leq \lim_{k\to\infty} \delta_k.
   \]
\end{thm}

\begin{proof}
    Because $\SSS_k$ are subspaces and by the measurability assumption, $P_{\SSS_k}$ is fixed given $W_{k-1}$.
    Therefore, one has
    \begin{equation}\label{ineq:proj_sq_bound}
    \begin{split}
        \mathbb{E}\big(\|P_{\SSS_k}(\tilde G_k)\|^2\big)
        &=
        \mathbb{E}\Big[\mathbb{E}\big(\|P_{\SSS_k}(\tilde G_k)\|^2 \mid W_{k-1}\big)\Big] \\
        &=
        \mathbb{E}\Big[\big\| \mathbb{E}\big(P_{\SSS_k}(\tilde G_k)\mid W_{k-1}\big)\big\|^2
        +\mathbb{V}\big(P_{\SSS_k}(\tilde G_k)\mid W_{k-1}\big)\Big] \\
        &=
        \mathbb{E}\Big[\|P_{\SSS_k}(\mathbb{E}[\tilde G_k\mid W_{k-1}])\|^2
        +\mathbb{E}\big(\|P_{\SSS_k}(\tilde G_k-\mathbb{E}[\tilde G_k\mid W_{k-1}])\|^2\mid W_{k-1}\big)\Big] \\
        &=
        \mathbb{E}\Big[\|P_{\SSS_k}(G_k)\|^2
        +\mathbb{E}\big(\|P_{\SSS_k}(\tilde G_k-G_k)\|^2\mid W_{k-1}\big)\Big] \\
        &\le
        \mathbb{E}\big[\|P_{\SSS_k}(G_k)\|^2\big]
        +\mathbb{E}\big[\|\tilde G_k-G_k\|^2\big] \\
        &\le
        \mathbb{E}\big[\|P_{\SSS_k}(G_k)\|^2\big]+C.
    \end{split}
    \end{equation}
    Here, $C$ comes from the bound in Assumption~\ref{assump:moments}.

    Similar to the deterministic case, we assume the frequency of the full gradient restart is $K$, and consider $k-1 \mod K = 0$ (i.e., when a full gradient restart occurs).
    Similarly, we define
    \begin{equation}
        \hat W_{k}:=W_{k-1}-\eta_kP_{\SSS_k}(\tilde G_k).
    \end{equation}
    By the property of the full gradient restart, we have
    \begin{equation}\label{ineq:descent_stochastic_restart}
    \begin{split}
     \mathbb{E}[\ell(\hat W_{k})]
     & \leq \mathbb{E}[\ell(W_{k-1})]+\mathbb{E}\big[\langle G_k,-\eta_k P_{\SSS_k}(\tilde G_k)\rangle \big]+\frac{L}{2}\mathbb{E}[\|\eta_k P_{\SSS_k}(\tilde G_k)\|^2] \\
     & = \mathbb{E}[\ell(W_{k-1})]-\eta_k\,\mathbb{E}\Big[\big\langle P_{\SSS_k}(G_k),\,\mathbb{E}[P_{\SSS_k}(\tilde G_k)\mid W_{k-1}]\big\rangle\Big]+\frac{L\eta_k^2}{2}\mathbb{E}[\|P_{\SSS_k}(\tilde G_k)\|^2]\\
     & = \mathbb{E}[\ell(W_{k-1})]-\eta_k\,\mathbb{E}\big[\|P_{\SSS_k}(G_k)\|^2\big]+\frac{L\eta_k^2}{2}\mathbb{E}[\|P_{\SSS_k}(\tilde G_k)\|^2]\\
     & \leq \mathbb{E}[\ell(W_{k-1})]-\eta_k\,\mathbb{E}\big[\|P_{\SSS_k}(G_k)\|^2\big]
     +\frac{L\eta_k^2}{2}\Big(\mathbb{E}\big[\|P_{\SSS_k}(G_k)\|^2\big]+C\Big)\\
     & =  \mathbb{E}[\ell(W_{k-1})]-\left(\eta_k-\frac{L\eta_k^2}{2}\right)\mathbb{E}\big[\|P_{\SSS_k}(G_k)\|^2\big]+\frac{L\eta_k^2}{2}C,
    \end{split}
    \end{equation}
    where the first inequality follows from Assumption~\ref{assump:loss_regular},
    the second equality uses linearity of projection and the tower property, and the third equality uses
    \[
        \mathbb{E}[P_{\SSS_k}(\tilde G_k)\mid W_{k-1}]
        =P_{\SSS_k}(\mathbb{E}[\tilde G_k\mid W_{k-1}])
        =P_{\SSS_k}(G_k),
    \]
    which follows from~\eqref{assump:unbiased_conditional} and measurability of $P_{\SSS_k}$.

    Then by Assumption~\ref{assump:descent_exp}, suppose $k-1\mod{K}=0$, and for iterates $i=k+1,\cdots,k+K-1$,
    \begin{equation}\label{ineq:descent_stochastic_regular}
        \mathbb{E}[\ell(W_{i})]\leq \mathbb{E}[\ell(W_{i-1})]+2C_i,
    \end{equation}
    where $2C_i$ comes from bounding the scenario where both {\tt Opt} and {\tt UpdateM} operate at $i$-th iterate.
    Then summing up the inequalities~\eqref{ineq:descent_stochastic_restart} and~\eqref{ineq:descent_stochastic_regular} for $i=k+1,\cdots,k+K-1$, and use the fact that $\mathbb{E}[\ell(W_k)]\leq \mathbb{E}[\ell(\hat W_k)]+C_k$ since $W_k$ is obtained by applying {\tt Opt} to $\hat W_k$, we have
    \begin{equation}
    \left(\eta_k-\frac{L\eta_k^2}{2}\right)\mathbb{E}[\|P_{\SSS_{k}}(G_{k})\|^2]-\frac{L\eta_k^2}{2}C-C_k-2\sum_{i=k+1}^{k+K-1} C_i\leq \mathbb{E}[\ell(W_{k-1})-\ell(W_{k+K-1})].
    \end{equation}
    By Assumption~\ref{assump:lr_stochastic}, $\eta_k\to 0$ so without loss of generality, we can assume $\frac{L\eta_k}{2}\leq \frac{1}{2}$ for any $k\in\mathbb{N}$.
    Therefore for all integer $k\in\mathbb{N}$, we have
    \begin{equation}
    \frac{\eta_{kK+1}}{2}\mathbb{E}[\|P_{\SSS_{kK+1}}(G_{kK+1})\|^2]-\frac{L\eta_{kK+1}^2}{2}C-C_{kK+1}-2\sum_{i=kK+2}^{(k+1)K}C_{i}\leq \mathbb{E}[\ell(W_{kK})-\ell(W_{(k+1)K})].
    \end{equation}
    By Assumption~\ref{assump:loss_regular}, there exists a constant $\ell_\star$ so that $\ell_\star\leq \ell(W)$ for any $W$.
    Summing up for $k\in\{1,\cdots, T\}$ one has
    \begin{equation}\label{ineq:proj_grad_bound}
    \begin{split}
        \sum_{k=1}^T\frac{\eta_{kK+1}}{2}\mathbb{E}[\|P_{\SSS_{kK+1}}(G_{kK+1})\|^2]  \leq &\frac{LC}{2}\sum_{k=1}^T\eta_{kK+1}^2+\sum_{k=1}^TC_{kK+1}+2\sum_{k=1}^T\sum_{i=kK+2}^{(k+1)K} C_i \\
        & +\mathbb{E}[\ell(W_{K})]-\ell_\star.
    \end{split}
    \end{equation}
    Taking $T\to\infty$, and note that $\sum_k\eta_k^2<\infty$ and $\sum_k |C_k|<\infty$, the first and second series on the right hand side of~\eqref{ineq:proj_grad_bound} are obviously bounded.
    For the third series, note that
    \begin{equation}
        \sum_{k=1}^\infty\sum_{i=kK+2}^{(k+1)K} C_i\leq \sum_{k=1}^\infty\sum_{i=kK+2}^{(k+1)K} |C_i|\leq \sum_{i=1}^\infty |C_i|<\infty.
    \end{equation}
    Finally $\|\mathbb{E}[\ell(W_{K})]-\ell_\star\|$ is obviously bounded for a fixed $K$, and therefore, one has
    \begin{equation}
        \sum_{k=1}^\infty \eta_{kK+1} \mathbb{E}[\|P_{\SSS_{kK+1}}(G_{kK+1})\|^2]<\infty.
    \end{equation}
    By $\sum_{k}^\infty\eta_k=\infty$ and a contradiction argument, one has
    \[
        \liminf_{k\to\infty}\mathbb{E}[\|P_{\SSS_k}(G_k)\|^2]=0
        \quad\Rightarrow\quad
        \liminf_{k\to\infty}\mathbb{E}[\|P_{\SSS_k}(G_k)\|]=0,
    \]
    where the implication follows from Jensen's inequality $\mathbb{E}\|X\|\le \sqrt{\mathbb{E}\|X\|^2}$.

    Since
    \[
        \|G_{k}\|\leq \operatorname{dist}(G_{k},\mathcal{S}_{k})+\|P_{\SSS_{k}}(G_{k})\|
        \leq \delta_k+\|P_{\SSS_{k}}(G_{k})\|,
    \]
    taking expectations and $\liminf$ yields
    \[
        \liminf_{k\to\infty}\mathbb{E}[\|G_k\|]
        \le \lim_{k\to\infty}\delta_k + \liminf_{k\to\infty}\mathbb{E}[\|P_{\SSS_k}(G_k)\|]
        = \lim_{k\to\infty}\delta_k,
    \]
    which completes the proof.
\end{proof}

\section{Non-asymptotic Rate: Proof of Theorem~\ref{thm:peso_rate}}\label{sec:proof_details}
\begin{algorithm}[t]
\caption{{\tt PESO-LoRA-R} (Analysis Version)}
\label{alg:peso_lora_r_proof}
\begin{algorithmic}[1]
\Require Pre-trained parameters $W_0\in\mathbb{R}^{m\times n}$, frequency $K$, rank $r$.
\Require Restart stepsizes $\{\eta_k\}$ (used only at restart steps), Adam stepsizes $\{\alpha_k\}$.
\Require Adam hyperparameters $(\beta_1,\beta_2,\epsilon)$.
\State Set $k \gets 1$, $\tilde W_0\gets W_0$, $A_0\gets 0$, $B_0\gets 0$.
\State Initialize Adam states $m_A\gets 0$, $v_A\gets 0$, $m_B\gets 0$, $v_B\gets 0$.
\While{stopping criteria not satisfied}
    \If{$(k-1)\bmod K = 0$}
        \State $\tilde W_k\gets \tilde W_{k-1}+A_{k-1}B_{k-1}$.
        \State Compute stochastic full gradient estimator $\tilde G_k$ at $\tilde W_k$.
        \State $(U_k,\Lambda_k,V_k)\gets {\tt SVD}(-\tilde G_k)$ 
        \State Set
        $A_{k-1}\gets \sqrt{\eta_k}\,U_k\sqrt{\Lambda_k}$,
        \quad
        $B_{k-1}\gets \sqrt{\eta_k}\,\sqrt{\Lambda_k}V_k^\top$.
        \State Reset Adam states: $m_A,v_A,m_B,v_B\gets 0$. 
    \Else
        \State $\tilde W_k\gets \tilde W_{k-1}$.
    \EndIf
    \State Compute stochastic gradients
    $g_{A,k} \gets \nabla_{A}\ell(\tilde W_k + A_{k-1}B_{k-1})$,
    $g_{B,k} \gets \nabla_{B}\ell(\tilde W_k + A_{k-1}B_{k-1})$.
    \State $m_A \gets \beta_1 m_A + (1-\beta_1) g_{A,k}$,\quad
           $v_A \gets \beta_2 v_A + (1-\beta_2) g_{A,k}\odot g_{A,k}$.
    \State $m_B \gets \beta_1 m_B + (1-\beta_1) g_{B,k}$,\quad
           $v_B \gets \beta_2 v_B + (1-\beta_2) g_{B,k}\odot g_{B,k}$.
    \State Bias corrections:
    $\hat m_A \gets m_A / (1-\beta_1^{k})$,\quad $\hat v_A \gets v_A / (1-\beta_2^{k})$,
    and similarly for $(\hat m_B,\hat v_B)$.
    \State Adam updates:
    $A_k \gets A_{k-1} - \alpha_k\, \hat m_A / (\sqrt{\hat v_A} + \epsilon)$.
    \State $B_k \gets B_{k-1} - \alpha_k\, \hat m_B / (\sqrt{\hat v_B} + \epsilon)$.
    \State $k \gets k+1$.
\EndWhile 
\State \Return $\tilde W_{k-1}+A_{k-1}B_{k-1}$.
\end{algorithmic}
\end{algorithm}

In this section, we provide the proof of the non-asymptotic convergence rate stated in Theorem~\ref{thm:peso_rate} for \texttt{PESO-LoRA-R}.
To facilitate the analysis, we focus on Algorithm~\ref{alg:peso_lora_r_proof}, which simplifies the practical implementation (Algorithm~\ref{alg:peso-lora-r-practical}).
Specifically, Algorithm~\ref{alg:peso_lora_r_proof} substitutes AdamW with standard Adam and employs an explicit reset of the momentum and velocity states at each restart, replacing the smoothing and alignment heuristics described in Appendix~\ref{sec:implementation}.

We follow the notation in the main text: $k$ denotes the global iteration index, and a restart happens when $(k-1)\bmod K=0$.
For conciseness in the proof, we index restart cycles by $s\in\{0,1,\dots\}$ and denote the restart iteration by
\[
k_s := sK+1 .
\]
We write $\tilde W_s := W_{k_s}$ for the weights before doing the restart $s$ (which is the updated anchored state).
Within cycle $s$, \texttt{PESO-LoRA-R} performs $K$ inner Adam steps producing $\tilde W_{s+1}$ at the next restart.
We assume Assumptions~\ref{assump:loss_regular}--\ref{assump:subspace} from the main text.
For clarity, we restate Assumption~\ref{assump:inner} and~\ref{assump:subspace} that require careful measurability.

\begin{assumption}[Detailed version of Assumption~\ref{assump:subspace}]
\label{assump:restart_eps}
Let $P_r(\cdot)$ denote the rank-$r$ truncated SVD operator mapping a matrix to its best rank-$r$ approximation in Frobenius norm.
Assume there exists a sequence $\{\varepsilon_s\}_{s\ge 0}$ such that at each restart iterate $k_s=sK+1$, 
\begin{equation}
\label{eq:eps_assump}
\big\|\mathbb{E}\!\left[P_r(\tilde G_{k_s})\mid \tilde W_s\right]-P_r(G_{k_s})\big\|_F \le \varepsilon_s,
\quad
\text{and}
\quad
\sum_{s=0}^{\infty}\eta_s\,\varepsilon_s^2 < \infty,
\end{equation}
where $\eta_s$ is the restart stepsize at cycle $s$.
\end{assumption}

\begin{assumption}[Detailed version of Assumption~\ref{assump:inner}]
\label{assump:inner_appendix}
For each restart cycle $s$, define the inner objective
$f_s(A,B):=\ell(\tilde W_s + AB)$ and let $X:=(A,B)\in\mathbb{R}^{m\times r}\times \mathbb{R}^{r\times n}$ with dimension $d:=r(m+n)$.
Let $\{X_{s,t}\}_{t=0}^{K}$ be generated by Adam on $f_s$ for $K$ steps with per-cycle reset $m_{s,0}=0$, $v_{s,0}=0$.
At each inner step $t\ge 1$, let $g_{s,t}$ be the stochastic gradient (with respect to $X$) used by Adam.
Let $\mathcal{F}_{s,t-1}$ denote the filtration generated by all random variables (gradients and iterates) up to inner step $t-1$ of cycle $s$.
Assume:
\begin{enumerate}
\item[(i)] $f_s$ has $L_{\rm in}$-Lipschitz continuous gradients;
\item[(ii)] $\mathbb{E}[g_{s,t}\mid \mathcal{F}_{s,t-1}] = \nabla f_s(X_{s,t-1})$;
\item[(iii)] $\|g_{s,t}\|_\infty \le G_\infty$ almost surely.
%\item[(iv)] $\mathbb{E}\!\left[\|g_{s,t}-\nabla f_s(X_{s,t-1})\|_2^2\mid \mathcal{F}_{s,t-1}\right]\le C_{\rm in}^2$.
\end{enumerate}
\end{assumption}

The proof proceeds in three parts: first, we establish the descent achieved by the restart step (Lemma~\ref{lem:outer_descent}); next, we bound the progress of the inner Adam updates on $X$ (Lemma~\ref{lem:adam_cycle_rigorous}); and finally, we combine these results to derive the convergence rate (Theorem~\ref{thm:rate_in_s}).

We begin with deriving a descent lemma for the outer restart update.
At restart $k_s=sK+1$, define the post-restart point
\begin{equation}
\label{eq:restart_update_def}
\hat W_s \;:=\;\tilde W_s - \eta_s\, P_r(\tilde G_{k_s}).
\end{equation}

\begin{lem}[Descent at restart]
\label{lem:outer_descent}
Let $\rho := r/p$ from Lemma~\ref{lem:twosided}.
Under Assumptions~\ref{assump:loss_regular}, \ref{assump:moments}, and \ref{assump:restart_eps}, for any $s\ge 0$,
\begin{equation}
\label{eq:outer_descent_final}
\mathbb{E}\big[\ell(\hat W_s)\big]
\;\le\;
\mathbb{E}\big[\ell(\tilde W_s)\big]
-\Bigl(\rho\eta_s-{L\eta_s^2}\Bigr)\,\mathbb{E}\|G_{k_s}\|_F^2
+\eta_s\,\mathbb{E}\big[\|G_{k_s}\|_F\,\varepsilon_s\big]
+{L\eta_s^2}\,C .
\end{equation}
\end{lem}

\begin{proof}
Fix $s$ and condition on $\tilde W_s$.
By $L$-smoothness of $\ell$ and the update rule \eqref{eq:restart_update_def},
\[
\ell(\hat W_s)
\le
\ell(\tilde W_s)
+\left\langle \nabla \ell(\tilde W_s),\, \hat W_s-\tilde W_s\right\rangle
+\frac{L}{2}\|\hat W_s-\tilde W_s\|_F^2
=
\ell(\tilde W_s)
-\eta_s\langle G_{k_s},\, P_r(\tilde G_{k_s})\rangle
+\frac{L\eta_s^2}{2}\|P_r(\tilde G_{k_s})\|_F^2.
\]
Taking the conditional expectation given $\tilde W_s$ yields
\begin{equation}
\label{eq:outer_condexp_1}
\mathbb{E}\big[\ell(\hat W_s)\mid \tilde W_s\big]
\le
\ell(\tilde W_s)
-\eta_s\left\langle G_{k_s},\, \mathbb{E}[P_r(\tilde G_{k_s})\mid \tilde W_s]\right\rangle
+\frac{L\eta_s^2}{2}\,\mathbb{E}\big[\|P_r(\tilde G_{k_s})\|_F^2\mid \tilde W_s\big].
\end{equation}
We first lower-bound the descent term: 
\begin{equation}
\begin{split}
\left\langle G_{k_s},\, \mathbb{E}[P_r(\tilde G_{k_s})\mid \tilde W_s]\right\rangle
& =
\langle G_{k_s},\, P_r(G_{k_s})\rangle
+\left\langle G_{k_s},\, \mathbb{E}[P_r(\tilde G_{k_s})\mid \tilde W_s]-P_r(G_{k_s})\right\rangle \\
& \ge \|P_r(G_{k_s})\|_F^2
- \|G_{k_s}\|_F\big\|\mathbb{E}[P_r(\tilde G_{k_s})\mid \tilde W_s]-P_r(G_{k_s})\big\|_F \\
& \ge \|P_r(G_{k_s})\|_F^2-\|G_{k_s}\|_F\,\varepsilon_s.
\end{split}
\end{equation}
where the first inequality uses $\langle G, P_r(G)\rangle = \|P_r(G)\|^2$ and Cauchy--Schwarz, and the second follows from Assumption~\ref{assump:restart_eps}.
Multiplying by $-\eta_s$ reverses the inequality, yielding
\begin{equation}
\label{eq:outer_descent_bound}
-\eta_s \left\langle G_{k_s},\, \mathbb{E}[P_r(\tilde G_{k_s})\mid \tilde W_s]\right\rangle
\le
-\eta_s \|P_r(G_{k_s})\|_F^2 + \eta_s \|G_{k_s}\|_F\,\varepsilon_s.
\end{equation}
Next, we bound the quadratic term:
\begin{equation}
\label{eq:outer_quad_bound}
\begin{split}
\frac{L\eta_s^2}{2}\,\mathbb{E}\big[\|P_r(\tilde G_{k_s})\|_F^2\mid \tilde W_s\big]
 & \le
\frac{L\eta_s^2}{2}\,\mathbb{E}\big[\|\tilde G_{k_s}\|_F^2\mid \tilde W_s\big] \\
& \le \frac{L\eta_s^2}{2}\,\left( 2\|G_{k_s}\|_F^2+ 2\mathbb{E}\big[\|\tilde G_{k_s}-G_{k_s}\|_F^2\mid \tilde W_s\big]\right) \\
& \le \frac{L\eta_s^2}{2}\,\left( 2\|G_{k_s}\|_F^2+ 2C\right)
\end{split}
\end{equation}
where the first inequality is from $\|P_r(X)\|_F\leq \|X\|_F$, the second is from Young's inequality, and the third is from Assumption~\ref{assump:moments}.
Plugging \eqref{eq:outer_descent_bound} and \eqref{eq:outer_quad_bound} into \eqref{eq:outer_condexp_1}, we obtain
\begin{equation}
\begin{split}
\mathbb{E}\big[\ell(\hat W_s)\mid \tilde W_s\big]
& \le
\ell(\tilde W_s)
-\eta_s\|P_r(G_{k_s})\|_F^2
+\eta_s\|G_{k_s}\|_F\,\varepsilon_s
+\frac{L\eta_s^2}{2}\left(2\|G_{k_s}\|_F^2 + 2C\right) \\
& \le
\ell(\tilde W_s)
-\Bigl(\rho\eta_s-{L\eta_s^2}\Bigr)\| G_{k_s} \|_F^2
+\eta_s\|G_{k_s}\|_F\,\varepsilon_s
+{L\eta_s^2}C,
\end{split}
\end{equation}
where the second inequality follows from Lemma~\ref{lem:twosided}.
Taking the total expectation completes the proof.
\end{proof}

We next show the progress made by Adam within one cycle.
Fix a cycle $s$, the inner Adam procedure starts from $X_{s,0}=(A_{s,0},B_{s,0})$ and after $K$ Adam steps yields $X_{s,K}$.
Following the indexing notation using $s$, the Adam updates take the form of 
\[
m_{s,t}=\beta_{1,t} m_{s,t-1}+(1-\beta_{1,t})g_{s,t},\qquad
v_{s,t}=\beta_2 v_{s,t-1}+(1-\beta_2)g_{s,t}^{\odot 2},
\]
with bias corrections $\hat m_{s,t}=m_{s,t}/(1-\prod_{j=1}^t\beta_{1,j})$,
$\hat v_{s,t}=v_{s,t}/(1-\beta_2^t)$, and
\[
X_{s,t}=X_{s,t-1}-\alpha_{s,t}\frac{\hat m_{s,t}}{\sqrt{\hat v_{s,t}}+\epsilon},
\qquad t=1,\dots,K.
\]

\begin{lem}[Adam progress bound]
\label{lem:adam_cycle_rigorous}
Assume Assumption~\ref{assump:inner_appendix}.
Fix a restart cycle $s$ and run $K$ steps of Adam with stepsizes $\alpha_{s,t}=a_s/\sqrt{t}$ for $t=1,\dots,K$, initialization $m_{s,0}=v_{s,0}=0$, and hyperparameters $\beta_2\in(0,1)$, $\beta_{1,t}\in[0,\bar\beta_1]$ with $\bar\beta_1 \in [0, 1)$.
Then
\begin{equation}
\label{eq:adam_cycle_bound_rigorous}
\mathbb{E}\big[f_s(X_{s,K})\big]
\le
\mathbb{E}\big[f_s(X_{s,0})\big]
+
A_1\,a_s
+
A_2\,a_s^2,
\end{equation}
where the constants are defined by
\begin{equation}
     A_1:= \frac{2d\,G_\infty^2}{\epsilon}\sum_{t=1}^K \frac{1}{\sqrt{t}},
     \qquad
     A_2:= \frac{L_{\rm in}}{2}B_1\sum_{t=1}^K \frac{1}{t}
\end{equation}
with $B_1 := \frac{d}{\epsilon^2}G_\infty^2$.
\end{lem}

\begin{proof}
Let the preconditioned step be defined as $D_{s,t} := \hat m_{s,t}/(\sqrt{\hat v_{s,t}}+\epsilon)$ so that $X_{s,t} = X_{s,t-1} - \alpha_{s,t}D_{s,t}$.
By $L_{\rm in}$-smoothness of $f_s$ and Taylor expansion
\begin{equation}
f_s(X_{s,t})
\le
f_s(X_{s,t-1})
-\alpha_{s,t}\langle \nabla f_s(X_{s,t-1}),\,D_{s,t}\rangle
+\frac{L_{\rm in}}{2}\alpha_{s,t}^2\|D_{s,t}\|_2^2.
\end{equation}
Taking total expectation and summing over $t=1,\dots,K$, we have
\begin{equation}
\label{eq:smooth_sum_main}
\mathbb{E}[f_s(X_{s,K})]
\le
\mathbb{E}[f_s(X_{s,0})]
-\sum_{t=1}^K \alpha_{s,t}\,\mathbb{E}\langle \nabla f_s(X_{s,t-1}), D_{s,t}\rangle
+\frac{L_{\rm in}}{2}\sum_{t=1}^K \alpha_{s,t}^2\,\mathbb{E}\|D_{s,t}\|_2^2.
\end{equation}
We first derive a bound for $\mathbb{E}\|D_{s,t}\|_2^2$.
We note that the bias-corrected momentum is $\hat m_{s,t}=\sum_{\tau=1}^t w_{t,\tau} g_{s,\tau}$ with $\sum_{\tau=1}^t w_{t,\tau}=1$ and $w_{t,\tau} \ge 0$, where
\[
w_{t,\tau}:=\frac{(1-\beta_{1,\tau})\prod_{j=\tau+1}^t \beta_{1,j}}{1-\prod_{j=1}^t\beta_{1,j}}\ge 0.
\]
Therefore $\hat m_{s,t}$ is a convex combination of past gradients and by Jensen's inequality and Assumption~\ref{assump:inner_appendix}(iii)
\[
\mathbb{E}\|\hat m_{s,t}\|_\infty^2 \le \sum_{\tau=1}^t w_{t,\tau} \mathbb{E}\|g_{s,\tau}\|_\infty^2 \le G_\infty^2.
\]
Since $\sqrt{\hat v_{s,t}}+\epsilon \ge \epsilon$ elementwise we then have
\begin{equation}
\label{eq:D_bound}
\mathbb{E}\|D_{s,t}\|_2^2
= \mathbb{E}\sum_{i=1}^d \frac{\hat m_{s,t,i}^2}{(\sqrt{\hat v_{s,t,i}}+\epsilon)^2}
\le \frac{1}{\epsilon^2}\mathbb{E}\|\hat m_{s,t}\|_2^2
\le \frac{d}{\epsilon^2}\mathbb{E}\|\hat m_{s,t}\|_\infty^2
\le B_1.
\end{equation}

We then denote the preconditioner as $H_{s,t}:=\mathrm{diag}(\sqrt{\hat v_{s,t}}+\epsilon)^{-1}$.
Consider the inner product term $\langle \nabla f_s(X_{s,t-1}), D_{s,t}\rangle$ and note that
\begin{equation}
\langle \nabla f_s(X_{s,t-1}), D_{s,t}\rangle
= \underbrace{\langle \nabla f_s(X_{s,t-1}), H_{s,t}\nabla f_s(X_{s,t-1})\rangle}_{(I)}
+ \underbrace{\langle \nabla f_s(X_{s,t-1}), H_{s,t}(\hat m_{s,t}-\nabla f_s(X_{s,t-1}))\rangle}_{(II)}.
\end{equation}
Since $H_{s,t} \succ 0$, $(I) \ge 0$, implying $-\alpha_{s,t} \mathbb{E}[(I)] \le 0$.

We next show $\|\nabla f_s\|_2 \le \sqrt{d}G_\infty$.
Let $\mathcal{F}_{s,t-1}$ denote the filtration generated by all random variables (gradients and iterates) up to inner step $t-1$ of cycle $s$.
By Assumption~\ref{assump:inner_appendix}(ii), we have $\nabla f_s(X_{s,t-1}) = \mathbb{E}[g_{s,t}\mid \mathcal{F}_{s,t-1}]$.
Using Jensen's inequality for the convex norm $\|\cdot\|_\infty$ and the almost sure bound $\|g_{s,t}\|_\infty \le G_\infty$ from Assumption~\ref{assump:inner_appendix}(iii),
\[
\|\nabla f_s(X_{s,t-1})\|_\infty
= \|\mathbb{E}[g_{s,t}\mid \mathcal{F}_{s,t-1}]\|_\infty
\le \mathbb{E}[\|g_{s,t}\|_\infty\mid \mathcal{F}_{s,t-1}]
\le G_\infty.
\]
Since $\|v\|_2 \le \sqrt{d}\|v\|_\infty$ for any $v \in \mathbb{R}^d$, we have
\begin{equation}
\label{eq:grad_norm_bound}
\|\nabla f_s(X_{s,t-1})\|_2 \le \sqrt{d}\,\|\nabla f_s(X_{s,t-1})\|_\infty \le \sqrt{d}G_\infty.
\end{equation}

For $(II)$, using Cauchy--Schwarz, $\|H_{s,t}\|_2 \le 1/\epsilon$, and $\|\nabla f_s\|_2 \le \sqrt{d}G_\infty$,
\begin{equation}
|(II)| \le \|\nabla f_s(X_{s,t-1})\|_2 \|H_{s,t}\|_2 \|\hat m_{s,t} - \nabla f_s(X_{s,t-1})\|_2
\le \frac{\sqrt{d}G_\infty}{\epsilon} \|\hat m_{s,t} - \nabla f_s(X_{s,t-1})\|_2.
\end{equation}
Moreover, by $\|\hat m_{s,t}\|_2 \le \sqrt{d}\|\hat m_{s,t}\|_\infty \le \sqrt{d}G_\infty$ and $\|\nabla f_s(X_{s,t-1})\|_2 \le \sqrt{d}G_\infty$, we have
\[
\|\hat m_{s,t} - \nabla f_s(X_{s,t-1})\|_2 \le \|\hat m_{s,t}\|_2 + \|\nabla f_s(X_{s,t-1})\|_2 \le 2\sqrt{d}\,G_\infty.
\]
Thus, we obtain a bound linear in $\alpha_{s,t}$:
\begin{equation}
-\alpha_{s,t}\mathbb{E}[(II)]
\le
\alpha_{s,t}\,\frac{\sqrt{d}G_\infty}{\epsilon}\,(2\sqrt{d}\,G_\infty)
=
\alpha_{s,t}\,\frac{2d\,G_\infty^2}{\epsilon}.
\end{equation}

Finally, substituting $\alpha_{s,t} = a_s/\sqrt{t}$ into \eqref{eq:smooth_sum_main}, we have
\begin{equation}
\begin{split}
\mathbb{E}[f_s(X_{s,K})]
&\le \mathbb{E}[f_s(X_{s,0})] + \sum_{t=1}^K \alpha_{s,t}\,\frac{2d\,G_\infty^2}{\epsilon} + \frac{L_{\rm in}}{2} \sum_{t=1}^K \alpha_{s,t}^2 B_1 \\
&= \mathbb{E}[f_s(X_{s,0})] + a_s \left( \frac{2d\,G_\infty^2}{\epsilon} \sum_{t=1}^K \frac{1}{\sqrt{t}} \right) + a_s^2 \left( \frac{L_{\rm in}}{2} B_1 \sum_{t=1}^K \frac{1}{t} \right) \\
&= \mathbb{E}[f_s(X_{s,0})] + A_1 a_s + A_2 a_s^2,
\end{split}
\end{equation}
according to the definition of $A_1$ and $A_2$.
This finishes the proof.
\end{proof}

\begin{thm}[Convergence rate in restart cycles]
\label{thm:rate_in_s}
Let $\ell_\star := \inf_W \ell(W) > -\infty$ and $\rho = r/p$ from Lemma~\ref{lem:twosided}.
Suppose Assumptions~\ref{assump:loss_regular}, \ref{assump:moments}, \ref{assump:restart_eps}, and \ref{assump:inner_appendix} hold.
Fix the total number of restart cycles $S \ge 1$, and choose the step sizes as $\eta_s=\frac{1}{\sqrt{S}}$ and $a_s=\frac{1}{S}$.
Set the hyperparameters of Adam as those in Lemma~\ref{lem:adam_cycle_rigorous}.
Then, the minimum expected gradient norm over the restart points satisfies
\begin{equation}
\label{eq:thm_rate_result}
\min_{0\le s\le S-1} \mathbb{E}\|G_{k_s}\|_F^2
\;\le\;
\frac{1}{\sqrt{S}}\left(\frac{4}{\rho} \left( \mathbb{E}[\ell(\tilde W_0)] - \ell_\star + \frac{1}{\rho}\bar{E}_{\rm sq} + LC + A_1+A_2 \right) \right),
\end{equation}
where $\bar{E}_{\rm sq} := \sum_{s=0}^{S-1} \eta_s \varepsilon_s^2$ (finite because of Assumption~\ref{assump:restart_eps}), $C$ from Assumption~\ref{assump:moments}, and $A_1, A_2$ from Lemma~\ref{lem:adam_cycle_rigorous}.
\end{thm}

\begin{proof}
We combine the descent progress from the restart step (Lemma~\ref{lem:outer_descent}) with the drift bound from the Adam inner steps (Lemma~\ref{lem:adam_cycle_rigorous}).
By Lemma~\ref{lem:outer_descent}, the update at restart $s$ satisfies:
\begin{equation}
\label{eq:proof_outer}
\mathbb{E}[\ell(\hat W_s)]
\le
\mathbb{E}[\ell(\tilde W_s)]
-\left(\rho\eta_s - L\eta_s^2\right)\mathbb{E}\|G_{k_s}\|_F^2
+\eta_s\,\mathbb{E}[\|G_{k_s}\|_F\varepsilon_s]
+L\eta_s^2 C.
\end{equation}

Without loss of generality, we assume $\eta_s \le \rho/2L$ (which holds for large $S$).
Therefore,
\[
\rho\eta_s - L\eta_s^2 \;\ge\; \frac{\rho}{2}\eta_s.
\]
Substituting this into \eqref{eq:proof_outer} gives
\begin{equation}
\label{eq:proof_outer_clean}
\mathbb{E}[\ell(\hat W_s)]
\le
\mathbb{E}[\ell(\tilde W_s)]
-\frac{\rho}{2}\eta_s\,\mathbb{E}\|G_{k_s}\|_F^2
+\eta_s\,\mathbb{E}[\|G_{k_s}\|_F\varepsilon_s]
+L\eta_s^2 C.
\end{equation}

By Young's inequality ($xy \le \frac{\rho}{4}x^2 + \frac{1}{\rho}y^2$ with $x=\sqrt{\eta_s}\|G_{k_s}\|_F$ and $y=\sqrt{\eta_s}\varepsilon_s$),
\begin{equation}
\mathbb{E}[\ell(\hat W_s)]
\le
\mathbb{E}[\ell(\tilde W_s)]
-\frac{\rho}{4}\eta_s\,\mathbb{E}\|G_{k_s}\|_F^2
+\frac{1}{\rho}\eta_s\varepsilon_s^2
+L\eta_s^2 C.
\end{equation}

By Lemma~\ref{lem:adam_cycle_rigorous}, identifying $f_s(X_{s,0}) \equiv \ell(\hat W_s)$ and $f_s(X_{s,K}) \equiv \ell(\tilde W_{s+1})$,
\begin{equation}
\label{eq:proof_inner}
\mathbb{E}[\ell(\tilde W_{s+1})]
\le
\mathbb{E}[\ell(\hat W_s)] + A_1 a_s + A_2 a_s^2.
\end{equation}

Combining these yields the recurrence for cycle $s$:
\begin{equation}
\mathbb{E}[\ell(\tilde W_{s+1})]
\le
\mathbb{E}[\ell(\tilde W_s)]
-\frac{\rho}{4}\eta_s\,\mathbb{E}\|G_{k_s}\|_F^2
+\frac{1}{\rho}\eta_s\varepsilon_s^2
+L\eta_s^2 C
+ A_1 a_s + A_2 a_s^2.
\end{equation}

Rearranging and summing over $s=0, \dots, S-1$ we have
\begin{equation}
\label{eq:general_sum_bound}
\frac{\rho}{4} \sum_{s=0}^{S-1} \eta_s \mathbb{E}\|G_{k_s}\|_F^2
\le
\mathbb{E}[\ell(\tilde W_0)] - \mathbb{E}[\ell(\tilde W_S)]
+\frac{1}{\rho}\sum_{s=0}^{S-1}\eta_s\varepsilon_s^2
+L C\sum_{s=0}^{S-1}\eta_s^2
+ \sum_{s=0}^{S-1} (A_1 a_s + A_2 a_s^2).
\end{equation}

We now substitute the specific step sizes $\eta_s = \frac{1}{\sqrt{S}}$ and $a_s = \frac{1}{S}$.
For the Adam error terms we have
\[
\sum_{s=0}^{S-1} (A_1 a_s + A_2 a_s^2)
= A_1 + \frac{A_2}{S}
\le A_1 + A_2.
\]
Substituting this back, and by $\ell(\tilde W_S) \ge \ell_\star$,
\[
\frac{\rho}{4\sqrt{S}} \sum_{s=0}^{S-1} \mathbb{E}\|G_{k_s}\|_F^2
\le
\mathbb{E}[\ell(\tilde W_0)] - \ell_\star
+ \frac{1}{\rho}\bar{E}_{\rm sq}
+ LC
+ A_1 + A_2.
\]
Finally, since $\sum_{s=0}^{S-1} \mathbb{E}\|G_{k_s}\|_F^2 \ge S \min_{s} \mathbb{E}\|G_{k_s}\|_F^2$, divide both sides by $\frac{\rho}{4}\sqrt{S}$ to get
\[
\min_{0\le s\le S-1} \mathbb{E}\|G_{k_s}\|_F^2
\le
\frac{4}{\rho \sqrt{S}} \left( \mathbb{E}[\ell(\tilde W_0)] - \ell_\star + LC + A_1+A_2 + \frac{1}{\rho}\bar{E}_{\rm sq} \right).
\]
This finishes the proof.
\end{proof}

Finally, we present the convergence rate in terms of the total number of iterations, as stated in Theorem~\ref{thm:peso_rate}, using the relationship that each restart cycle consists of $K$ iterations.

\begin{cor}[Convergence rate in total iterations]
\label{cor:rate_in_T}
Let $T$ be the total number of iterations allowed.
Let $K$ be the fixed number of Adam steps per restart cycle, such that the number of restart cycles is $S = \lfloor T/K \rfloor$.
Under the conditions of Theorem~\ref{thm:rate_in_s}, the convergence rate with respect to the minimum gradient norm over all $T$ iterates is:
\begin{equation}
\min_{1\le k\le T} \mathbb{E}\|G_k\|_F^2
\;\le\;
\frac{C_0}{\sqrt{\lfloor T/K \rfloor}}
\;=\;
\mathcal{O}\left(\frac{1}{\sqrt{T}}\right),
\end{equation}
where $C_0 := \frac{4}{\rho} \left( \mathbb{E}[\ell(\tilde W_0)] - \ell_\star + LC + (A_1+A_2) + \frac{1}{\rho}\bar{E}_{\rm sq} \right)$.
\end{cor}

\begin{proof}
Let $\mathcal{S} = \{k_0, k_1, \dots, k_{S-1}\}$ be the set of indices corresponding to restart points. Since the set of restart points is a subset of all iterates $\{1, \dots, T\}$, the minimum over all iterates is upper bounded by the minimum over the restart points:
\[
\min_{1\le k\le T} \mathbb{E}\|G_k\|_F^2
\;\le\;
\min_{s \in \{0, \dots, S-1\}} \mathbb{E}\|G_{k_s}\|_F^2.
\]
Substituting $S = \lfloor T/K \rfloor$ into the bound from Theorem~\ref{thm:rate_in_s}:
\[
\min_{0\le s\le S-1} \mathbb{E}\|G_{k_s}\|_F^2
\le
\frac{C_0}{\sqrt{S}}
= \frac{C_0}{\sqrt{\lfloor T/K \rfloor}}.
\]
For large $T$, we have $\sqrt{\lfloor T/K \rfloor} \approx \sqrt{T/K}$. Thus, the bound simplifies to:
\[
\frac{C_0}{\sqrt{T/K}}
= \frac{C_0 \sqrt{K}}{\sqrt{T}}.
\]
Thus, the rate for the iterates is $\mathcal{O}(1/\sqrt{T})$.
\end{proof}

\end{document}